\documentclass[final]{siamltex}
\usepackage{amssymb, amsmath}
\usepackage{float,epsfig}
\usepackage{algpseudocode}
\usepackage{algorithm}

\newtheorem{remark}[theorem]{ Remark}
\newtheorem{example}{Example}[section]

\newcommand{\ba}{\begin{array}}
\newcommand{\ea}{\end{array}}

\newcommand{\be}{\begin{equation}}
\newcommand{\ee}{\end{equation}}
\newcommand{\beano}{\begin{eqnarray*}}
\newcommand{\eeano}{\end{eqnarray*}}

\def\R{{\mathbb R}}
\def\C{{\mathbb C}}

\def\lam{\lambda}
\def\sig{\sigma}

\def\diag{\mathrm{diag}}

\def\rank{\mathrm{rank}}

\def\rar{\rightarrow}

\def \nrank{\mathrm{nrk}}

\def \sp{\mathrm{Sp}}

\def \R{{\mathbb R}}
\def \C{{\mathbb C}}

\def \sig{\sigma}

\def \C{{\mathbb C}}

\def \lam{\lambda}
\def \sig{\sigma}

\def \diag{\mathrm{diag}}

\def \rank{\mathrm{rank}}

\def \nrank{\mathrm{nrank}}

\def \sp{\textrm{Sp}}

\def \rar{\rightarrow}

\title{Fiedler linearizations of multivariable state-space system and its associated system matrix}

\author{ Namita Behera  \thanks{Department of Mathematics, Sikkim University,
Sikkim-737102, INDIA,({\tt niku.namita@gmail.com, nbehera@cus.ac.in})}
\and  Avisek Bist \thanks{Department of Mathematics, Sikkim University,
Sikkim-737102, INDIA, ({\tt abist.21pdmt01@sikkimuniversity.ac.in, avisek.bista@gmail.com})}}

\begin{document}
\maketitle

\begin{abstract} Linearization is a standard method in the computation of eigenvalues and eigenvectors  of matrix polynomials. In the last decade a variety of linearization methods have been developed in order to deal with algebraic structures and in order to construct efficient numerical methods. An important source of linearizations for matrix polynomials are the so called  \emph{Fiedler pencils}, which are generalizations of the Frobenius companion form and these linearizations have been extended to regular rational matrix function which is the transfer function of LTI State-space system in \cite{rafinami,behera}. We consider a  multivariable state-space system and its associated system matrix $\mathcal{S}(\lam).$  We introduce Fiedler pencils of $\mathcal{S}(\lam)$ and describe
an algorithm for their construction. We show that Fiedler pencils are linearizations of the system matrix $\mathcal{S}(\lam)$.
%We study recovery of zero directions of multivariable state-space system from those of the linearizations. That is, the zero directions of the transfer functions associated to multivariable state-space system are recovered from the eigenvectors of the Fiedler pencils without performing any arithmetic operations.
%\mpar{to rewrite}
%The new technique is applied to Rosenbrock functions arising in mathematical system theory.
\end{abstract}
%
%\noindent



%\renewcommand{\thefootnote}{\fnsymbol{footnote}}
%\footnotetext[1]{
%Department of Mathematics, Sikkim University, Sikkim-737102, India.
%\texttt{nbehera@cus.ac.in}, \texttt{niku.namita@gmail.com}.
%}
%
%\footnotetext[2]{
%Department of Mathematics, Sikkim University, Sikkim-737102, India.
%\texttt{abist.21pdmt@sikkimuniversity.ac.in}.
%}
%
%
%\renewcommand{\thefootnote}{\arabic{footnote}}

\section{Introduction}
We denote by $\C[\lam]$ the polynomial ring over the complex field $\C.$  Further, we denote by $ \C^{m\times n}$ and $  \C[\lam]^{m\times n}$, respectively, the vector spaces of $m\times n$ matrices and  matrix polynomials over $\C.$

Consider a matrix polynomial $P(\lam) =\sum_{j=0}^{m} \lam^{j}A_j , \,\, A_j \in \C^{n \times n}$. Then  a matrix polynomial $P(\lam)$ is said to be regular if $\det(P(\lam)) \neq 0$ for some $\lam \in \C$. Linearization is a standard method for solving polynomial eigenvalue problems $P(\lam) x = 0$. Let $ P(\lam)$ be an $n\times n$ matrix polynomial (regular or singular) of degree $m.$ Then an $ mn\times mn$  matrix pencil $L(\lam) := X+ \lam Y$ is said to be a {\em linearization}~\cite{gohberg82} of  $P(\lam)$ if there are $mn\times mn$ unimodular matrix polynomials $U(\lam)$ (the determinant of $U(\lam),$ is a nonzero constant for all $\lam \in \mathbb{C}.$) and $V(\lam)$ such that $$  U(\lam) L(\lam) V(\lam) = \diag(I_{(m-1)n}, \,\, P(\lam)) $$   for all $\lam \in \C,$ where $I_k$ denotes the $k\times k$ identity matrix. Linearizations of matrix polynomials have been studied extensively over the years, see~\cite{gohberg82, mmmm06} and references therein. Recently, a new family of linearizations of matrix polynomials referred to as Fiedler linearizations (or {\em Fiedler pencils})  has been introduced and is an active area of research, see~\cite{AV04, TDM} and references therein. One of the distinctive features of a Fiedler pencil $L(\lam)$ of the matrix polynomial $P(\lam)$ is that its construction is operation free, that is, block entries of $L(\lam)$ are either $0$ or $ I_n$ or the coefficient matrices of $P(\lam)$,  and that $L(\lam)$ allows an easy (operation free) recovery of eigenvectors of $P(\lam)$ from the eigenvectors of $L(\lam)$~\cite{TDM, AV04}.

In this paper we extend the concept of Fiedler linearization from LTI state-space system to general multivariable state-space system and associated system matrix. In particular, in this paper we discuss the solution (finding \emph{eigenvalues} $\lam \in \mathbb C$ and \emph{eigenvectors} $v \in \mathbb C^n$)  of \emph{multivariable state-space system $\Sigma$}
\begin{equation}
\begin{aligned}
A\left(\frac{d}{dt}\right) x(t) &= B u(t), \\
y(t) &= C x(t) + D\left(\frac{d}{dt}\right) u(t) \,\,\,\, t\geq 0,
\end{aligned} \label{msss}
\end{equation}
such that $\mathcal{S}(\lam)v = 0, $
where $A(\lam) = \sum_{j=0}^{d_{A}}\lam^{j}A_j \in \C[\lam]^{n \times n}$ is a regular matrix polynomial of degree $d_{A}$, $D(\lam) = \sum_{j=0}^{d_{D}}\lam^{j}D_j \in \C[\lam]^{m \times m}$ is a matrix polynomial of degree $d_{D}$, and $ C \in \C^{m \times n }, B \in \C^{n \times m}$,
and its associate Rosenbrock system matrix $\mathcal{S}(\lam)$,
\begin{equation}
{S}(\lambda) = \left[ {\begin{array}{c|c}
A(\lambda) & -B \\
\hline
C & D(\lambda) \\
\end{array}}\right] {\mathbb C}[\lambda]^{n+m,n+m}\label{smpq}
\end{equation}
and the associated transfer function
\begin{equation}
R(\lambda) = D(\lambda) + C A(\lambda)^{-1}B \in {\mathbb C}(\lambda)^{m,m}. \label{tf}
\end{equation}

%Our aim is to study linearizations of $R(\lambda)$ and its relation to linearizations of $S$ in (\ref{smpq}).
%\begin{equation}\label{ratevp}
%R(\lambda) = A(\lambda) + C D(\lambda)^{-1}B \in {\mathbb F}(\lambda)^{n,n},
%\end{equation}
%where for $A(\lambda) = \sum_{i=0}^{d_A} \lambda^{i}A_i\in {\mathbb F}[\lambda]^{n,n}$  is regular and
%$D(\lambda) = \sum_{j=0}^{d_D}\lambda^j D_j\in {\mathbb F}[\lambda]^{m,m}$ and $C, B$ are constant %matrices of appropriate dimensions.

%By $\mathbb F(\lam)^{p,m}$ we denote the $p\times m$ matrices with entries that are rational functions with coefficients in the field $\mathbb F$, where $\mathbb F=\mathbb R$ or $\mathbb F=\mathbb C$ denotes the field of real or complex numbers, respectively.

%Denoting the \emph{ normal rank} of $R\in \mathbb F(\lam)$ as the rank in the set of rational matrix valued functions, we say that $\lambda_0$ is an \emph{eigenvalue} of $R(\lam)$ if the rank of $R(\lambda_0)$ is smaller than the normal rank, and $u_0$ is an eigenvector associated with $\lambda_0$ if it is in the kernel of $R(\lambda_0)$.

Next, consider a more general linear multivariable time invariant state-space system $\Sigma_1$ on the positive half line $\R_{+}$ in the representation
\begin{eqnarray}
 0&=& A\left(\frac{d}{dt}\right) x(t) +B\left(\frac{d}{dt}\right) u(t), \nonumber \\
 y(t)  &=& C\left(\frac{d}{dt}\right) x(t) + D\left(\frac{d}{dt}\right) u(t).\label{lti-system}
\end{eqnarray}
The function $u : \R^{+} \rightarrow \R^{m}$ is the input vector, $x: \R_{+} \rightarrow \R^{r}$ is the state vector, $y : \R_{+} \rightarrow \R^{p}$ is the output vector, and for $M(\lam) =\sum_{i=0}^d M_i \lambda^i\in \mathbb{C}[\lam]^{p, p}$ we use  $M(\frac{d}{dt})$ to denote the differential operator $\sum_{i=0}^\ell M_i \frac{d^i}{dt^i}$, where $\frac{d}{dt}$ denotes time-differentiation.

The associated matrix polynomial is
\begin{equation}\label{rosmatrix}
S(\lam ) := \left[
                            \begin{array}{c|c}
                              A(\lam) & B(\lam) \\
                              \hline
                              C(\lam) & D(\lam) \\
                            \end{array}
                          \right] \in \C[\lam]^{(p+m),(p+m)}.
\end{equation}
The associate transfer function is defined by
\begin{equation}\label{trfun}
R(\lam) := D(\lam) - C(\lam) A(\lam)^{-1} B(\lam)  \in \C(\lam)^{p, p},
\end{equation}
where,  denoting by $\mathbb C[\lam]^{p, p}$ the vector space of $p \times p$ matrix polynomials, we assume that $ A(\lam) \in \mathbb{C}[\lam]^{m,m}$, $B(\lam) \in \mathbb{C}[\lam]^{m, p}$, $ C(\lam) \in \mathbb{C}[\lam]^{p, m}$, $D(\lam) \in \mathbb{C}[\lam]^{p,p}$.

Notice that in (\ref{smpq}) we consider $B$ and $C$ are constant matrices.

Rational eigenvalue problems arise in many applications, see e.g. \cite{planchard89,voss06,planchard,voss03} and the references therein. Rational matrix value functions of this form arise e.g. in linear system theory, see e.g. \cite{vardulakis}.

If $A(\lambda)$ is  regular, i.e., $\det A(\lambda)$ does not vanish identically,  then performing a Schur complement, one obtains the rational matrix function (\ref{trfun})
which, in frequency domain, describes the \emph{transfer function} from the Laplace transformed input to the Laplace transformed output of the system. In this case $S(\lambda)$ is called a \emph{Rosenbrock system matrix}, see \cite{rosenbrock70}.
%Rosenbrock systems arise from linearization of nonlinear systems around stationary solutions \cite{Cam95}, or from discretization of %distributed parameter systems ~\cite{???}. \mpar{more examples here, and a true evp application}

Conversely, if one has a given  rational matric function of the form
(\ref{trfun}), then one can always interpret it as originating
from a Rosenbrock system matrix of the form (\ref{rosmatrix}). Such rational matrix valued functions arise from realizations of input-output data, see e.g. \cite{MayA07}, or in model order reduction, see e.g. \cite{Ant05,GugA04}.

We consider the general square polynomial eigenvalue problem
\begin{equation}\label{rosevp}
S(\lam_0 )\left[
                            \begin{array}{c}
                            x_0 \\ u_0
                            \end{array}
                          \right]:= \left[
                            \begin{array}{c|c}
                              A(\lam) & B(\lam) \\
                              \hline
                              C(\lam) & D(\lam) \\
                            \end{array}
                          \right] \left[
                            \begin{array}{c}
                            x_0 \\ u_0
                            \end{array}
                          \right] = 0
%S(\lam_0 ) \mat{c} x_0 \\ u_0 \rix:= \mat{cc}
%                              A(\lam) & B(\lam)\\
%                              C(\lam) & D(\lam) \rix \mat{c} x_0 \\ u_0 \rix =0
\end{equation}

If $A$ and $D$  are square and regular, then one can form the rational function
$R(\lambda)$ as in (\ref{trfun})
% and
%
%\begin{equation}\label{grfun}
%G(\lambda)= A(\lambda)-B(\lambda)D(\lambda)^{-1}C(\lambda),
%\end{equation}
%
and,  since $\det S(\lambda)=\det A(\lambda)\det R(\lambda)$,
%= \det D(\lambda) \det G(\lambda)$,
it is clear that the eigenvalues of $S(\lambda)$ are the eigenvalues  of $A$ and $R$ combined
%, or those of $D$ and $G$ combined,
and the eigenvalues of $A(\lambda)$ are the poles of $R$.
%, while the eigenvalues of $D$ are the poles of $G$.
%But for non-square $D$ the situation is more involved and to discuss this question is one of the topics of this paper.
We restrict ourselves to rational functions of the form (\ref{trfun}) with regular $A(\lambda)$ and we assume for simplicity that $B,C$ are constant matrices in $\lambda$. All the results can be extended (with a lot of technicalities) to the case that $B,C$ depend on $\lambda$.

The system $\Sigma_1$ given in (\ref{lti-system}) is said to be in {\em state-space form} if it is given by
\begin{equation}
\begin{aligned}
E\dot{x}(t) &= A x(t) + B u(t) \\
y(t) &= C x(t) + P(\lam) u(t),
\end{aligned}\label{tils}
\end{equation}
where $P(\lam) \in \C[\lam]^{n\times n}$ is a matrix polynomial and $A, E \in \C^{m\times m}$ with $E$ being nonsingular, $ B \in \C^{n \times m}, C \in \C^{m \times n}$ are constant matrices, see~\cite{vardulakis}.
% We denote the system defined in (\ref{tils}) denoted by $(E, A, B, C, P(\lam))$.
For linear time invariant (LTI) state-space system given in (\ref{tils}), there is a state-space framework developed in \cite{rafinami} to study the zeros of LTI system in state space form.

For computing zeros of a linear time-invariant system $\Sigma$ in state-space form, it has been introduced Fiedler-like pencils and Rosenbrock linearization of the Rosenbrock system
polynomial $\mathcal{S}(\lam)$ associated with $\Sigma$. Also, it has shown that  the Fiedler-like pencils are Rosenbrock linearizations of the system polynomial $\mathcal{S}(\lam)$, see \cite{rafinami, rafinami2, rafinami3, behera20}

Next, for the higher order linear time invariant (LTI) state-space system $\Sigma_1$ given by
\begin{equation}
\begin{aligned}
P\left(\frac{d}{dt}\right) x(t) &= B u(t), \\
y(t) &= C x(t) + D u(t),
\end{aligned} \label{mulsss}
\end{equation}
where $P(\lam) = \sum_{j=0}^{m}\lam^{j}A_j \in \C[\lam]^{n \times n}$ is regular matrix polynomial of degree $m$ and $ D \in \C^{r \times r}, C \in \C^{r \times n}, B \in \C^{n \times r}$, there is a state-space framework developed in \cite{beh21} to study the zeros of $\Sigma_1$.
%The associated system matrix is given by
%\begin{equation}
%\mathcal{S}(\lam) = \left[
%                       \begin{array}{c|c}
%                         P(\lam) & -B \\
%                         \hline
%                         C & D \\
%                       \end{array}
%                     \right] \in \C^{(n+r)\times (n+r)} \label{smh}
%\end{equation}
%and the transfer function of $\Sigma_1$ is given by
%\begin{equation}
%G(\lam) = C P(\lam)^{-1} B + D \in \C^{r \times r}. \label{tfh}
%\end{equation}
To study the eigenvalues and eigenvectors of the system matrix associated with $\Sigma_1$ recently,
it has been introduced Fiedler-like pencils and Rosenbrock linearization system
polynomial  associated with $\Sigma_1$. Also, it has shown that  the Fiedler-like pencils are Rosenbrock linearizations of the system polynomial, see \cite{beh21}

In this paper we study the relationship between the eigenvalues and eigenvectors of a rational eigenvalue problem given in the form of a transfer function (\ref{trfun}), its polynomial representation as a Rosenbrock matrix  and associated linearizations. we introduce Fiedler linearizations  of the system
matrix $\mathcal{S} (\lam)$ given in (\ref{smpq}) which is also helpful to study zeros of the system $\Sigma$ given in (\ref{msss}). This problem has recently has been studied for higher order state-space system in \cite{beh21} and we will extend these results to the Multivariable state-space case.

The paper is organized as follows.  In section~\ref{sec:Bc} we recall the definition and some properties of matrix which we need throughout this paper. In section~\ref{sec:FP}  we extend the results of Fiedler pencils for Rosenbrock system from \cite{rafinami, behera} to multivariable state space system. That is, we show that the state-space framework so developed in \cite{rafinami} could be
gainfully used to linearize (Fiedler linearizations) a multivariable state-space
system. In particular, we define Fiedler pencils for $\mathcal{S} (\lam)$ given in (\ref{smpq}) and describe an algorithm for their construction.  In Section~4 we prove that these are linearizations for associate system matrix $\mathcal{S} (\lam)$.

{\bf Notation.} An $m\times n$ rational matrix function $R(\lam)$ is an $m\times n$ matrix whose entries are rational functions of the form $p(\lam)/{q(\lam)},$ where $p(\lam)$ and $q(\lam)$ are scalar polynomials in $\C[\lam].$ We denote  the $j$-th column of the $n \times n$ identity matrix $I_n$ by $e_j$ and the transpose of a matrix $A$ by $A^T.$

\section{Basic Concepts}\label{sec:Bc}
\begin{definition} \cite{Ste97}
Let $A \in  \C^{m\times n}, B \in \C^{p\times q}.$ Then the Kronecker product (tensor product) of $A$ and $B$ is defined by
$$A \otimes B = \left[
    \begin{array}{ccc}
      a_{11}B & \cdots &  a_{1n}B \\
      \vdots & \ddots & \vdots \\
       a_{m1}B & \cdots &  a_{mm}B \\
    \end{array}
  \right] \in  \C^{mp\times nq}.
$$
\end{definition}
One of the properties of Kronecker product is as follows: Let $A \in  \C^{m\times n}, B \in \C^{r\times s}, C \in  \C^{n\times p}, D \in \C^{s\times t}$. Then $(A \otimes B)(C \otimes D) = (AC \otimes BD)\in  \C^{mr\times pt}.$

In order to systematically generate the Fiedler linearizations for Rosenbrock system matrices, we need a few concepts introduced in \cite{TDM}.
\begin{definition}\label{dci}
Let $\sigma : \{0, 1, \ldots, p-1\} \rightarrow \{1, 2, \ldots, p\}$ be a bijection.
\begin{itemize}
\item[(1)] For $j = 0, \ldots, p-2$, the bijection is said to have  a consecution at $j$ if $\sigma(j) < \sigma(j+1)$ and $\sigma$ has an inversion at $j$ if $\sigma(j) > \sigma(j+1)$.
\item[(2)] The tuple CISS$(\sigma) :=(c_{1}, i_{1}, c_{2}, i_{2}, \ldots, c_{l}, i_{l})$ is called the consecution-inversion structure sequence of $\sigma$, where $\sigma$ has $c_{1}$ consecutive consecutions at $0, 1, \ldots, c_{1}-1;$ $i_{1}$ consecutive inversions at $c_{1}, c_{1}+1, \ldots, c_{1}+i_{1}-1$ and so on, up to $i_{l}$ inversions at $p-1-i_{l}, \ldots, p-2$.
\item[(3)] The total number of consecutions and inversions in $\sigma$ is denoted by $c(\sigma)$ and $i(\sigma)$, respectively, i.e., $c(\sig) = \sum\limits_{j=1}^{l}c_{j}$,  $i(\sig) = \sum\limits_{j=1}^{l}i_j$, and $c(\sig) + i(\sig) = p-1$.
\end{itemize}
\end{definition}

\begin{definition}~\cite{TDM}
Let $P(\lambda) = A_{0} + \lambda A_{1} + \cdots + \lambda^{m}A_{m}$ be a matrix polynomial of degree $m$. For $k = 0, \ldots m$, the degree $k$ Horner shift of $P(\lambda)$ is the matrix polynomial $P_{k}(\lambda) := A_{m-k} + \lambda A_{m-k+1} + \cdots + \lambda^{k}A_{m}$. These Horner shifts satisfy the following:
\begin{align}
& P_{0}(\lambda) = A_{m}, \nonumber  \,\,\,
 P_{k+1}(\lambda) = \lambda P_{k}(\lambda) + A_{m-k-1} , \mbox{  for } 0\leq k\leq m-1, \,\,\,
P_{m}(\lambda) = P(\lambda). \nonumber \label{hs}
\end{align}
\end{definition}

\begin{definition}\cite{mmmm06}
Matrix polynomial $U(\lam)$ is said to be unimodular if $\det U(\lam)$ is a nonzero constant, independent of $\lam$. Two matrix
polynomials $P(\lam)$ and $Q(\lam)$ are said to be equivalent if there exists unimodular matrices $U(\lam)$ and $V(\lam)$, such that $Q(\lam) = U(\lam)P(\lam)V (\lam).$ If $U(\lam), V (\lam)$ are constant matrices, then $P(\lam)$ and $Q(\lam)$ are said to be strictly equivalent.
\end{definition}

Let $X(\lam)\in \mathbb{C}(\lam)^{m \times n}$ be a rational matrix function. The normal rank~\cite{rosenbrock70} of $ X(\lam)$ is denoted by $ \nrank(X)$ and is given by $\nrank(X) := \max_{\lam}\rank(X(\lam))$ where the maximum is taken over all $\lam \in \C$ which are not poles of the entries of $X(\lam).$ If $\nrank(X) =n = m$ then $X(\lam)$ is said to be regular, otherwise $X(\lam)$ is said to be singular.

A complex number $\lam$ is said to be an eigenvalue of the system matrix $\mathcal{S}(\lam)$ if
$\rank(\mathcal{S}(\lam)) < \nrank(\mathcal{S}).$ An eigenvalue $\lam$ of $\mathcal{S}(\lam)$ is called an invariant zero
of the system $\Sigma$. We denote the set of eigenvalues of $\mathcal{S}(\lam)$ by $\sp(\mathcal{S}).$

Let $G(\lam)\in \mathbb{C}(\lam)^{m \times n}$ be a rational matrix function and let
\begin{equation*}
\mathbf{SM}(G(\lam)) = \diag\left( \frac{\phi_{1}(\lam)}{\psi_{1}(\lam)}, \cdots, \frac{\phi_{k}(\lam)}{\psi_{k}(\lam)}, 0_{m-k, n-k}\right)
\end{equation*}
be the Smith-McMillan form~\cite{kailath, rosenbrock70} of  $G(\lam),$
where the scalar polynomials $\phi_{i}(\lam)$ and $ \psi_{i}(\lam)$ are monic, are pairwise coprime and,  $\phi_{i}(\lam)$ divides $\phi_{i+1}(\lam)$ and $\psi_{i+1}(\lam)$ divides $\psi_{i}(\lam),$ for $i= 1, 2, \ldots, k-1$.
The polynomials $\phi_1(\lam), \ldots, \phi_k(\lam)$ and $ \psi_1(\lam), \ldots, \psi_k(\lam)$ are called {\em invariant zero polynomials} and {\em invariant pole polynomials} of $G(\lam),$ respectively. Define
$$\phi_{G}(\lam) := \prod _{j=1}^{k} \phi_{j}(\lam) \,\,\, \mbox{ and } \,\,\, \psi_{G}(\lam) := \prod _{j=1}^{k} \psi_{j}(\lam).$$
 A complex number $ \lam $ is said to be a  zero of $G(\lam)$ if $ \phi_G(\lam) =0$ and
a complex number $ \lam$ is said to be a  pole of $G(\lam)$ if $\psi_G(\lam) =0.$ The {\bf spectrum} $\sp(G)$ of $G(\lam)$ is given by  $ \sp(G) :=\{ \lam \in \C : \phi_G(\lam) = 0\}.$ That is $\sp(G)$ is the set of zeros of $G(\lam),$ see \cite{rafinami}.

%The right and the left null spaces of $G(\lam)$, denoted by $ \mathcal{N}_r(G)$ and $\mathcal{N}_l(G)$ respectively, are given by $ \mathcal{N}_r(G) := \{ x(\lam) \in \C(\lam)^n: G(\lam)x(\lam) = 0\}$ and $\mathcal{N}_l(G) :=\{ y(\lam) \in \C(\lam)^m : y^T(\lam) G(\lam) =0 \}.$

\section{Fiedler pencils for Rosenbrock system matrix}\label{sec:FP}
In this section we define Fiedler pencils for system polynomial $\mathcal{S}(\lambda)$ and describe
an algorithm for their construction.
Let us consider a Rosenbrock system of the form (\ref{smpq}) with $B,C$ constant in $\lambda$,
\begin{equation}
\mathcal{S}(\lambda) = \left[ {\begin{array}{c|c}
A(\lambda) & -B \\
\hline
C & D(\lambda) \\
\end{array}}\right] {\mathbb C}[\lambda]^{n+m,n+m}
\end{equation}
and the associated transfer function
\[
R(\lambda) = D(\lambda) + C A(\lambda)^{-1}B \in {\mathbb F}(\lambda)^{m,m}
 \]
where for $A(\lambda) = \sum_{i=0}^{d_A} \lambda^{i}A_i\in {\mathbb C}[\lambda]^{n,n}$  is regular and
$D(\lambda) = \sum_{j=0}^{d_D}\lambda^j D_j\in {\mathbb C}[\lambda]^{m,m}$.
Our aim is to study linearizations of $R(\lambda)$ and its relation to linearizations of $\mathcal{S}(\lam)$.

The most simple way to perform a direct linearization is to consider a first companion form
\begin{equation}
\mathcal{C}_1(\lam)w := (\lam X+Y) w = 0, \label{C1}
\end{equation}
where
\[
 X := \left[
    \begin{array}{cccc|cccc}
      A_{d_A} &  &  &  &  & & & \\
       & I_n &  &  &  & & & \\
       &  & \ddots &  &  & & & \\
       &  &  & I_n &  & & & \\
      \hline
       &  &  &  & D_{d_D} & & & \\
       &  &  &  &  & I_m & & \\
       &  &  &  &  & & \ddots & \\
       &  &  &  &  &  & & I_m \\
    \end{array}
  \right],\]
\[
Y := \left[\begin{array}{cccc|cccc}
A_{d_A-1} & A_{d_A-2} & \cdots & A_0 & 0 & \cdots & 0 & -B \\
   -I_n & 0 & \cdots & 0 &  & 0 &  & 0 \\
    & \ddots & \ddots & & &  & \ddots & \vdots \\
    &  & -I_n & 0 &  &  &  & 0 \\
    \hline
    0 & \cdots & 0 & C & D_{d_D-1} & D_{d_D-2} & \cdots & D_0 \\
    & 0 &  & 0 & -I_m & 0 & \cdots & 0 \\
    &  & \ddots & \vdots &  & \ddots & \ddots & \vdots \\
    &  &  & 0 &  &  & -I_m & 0 \\
    \end{array}
\right],
\]
and
$$w = \left[
        \begin{array}{c}
          \lam^{d_A-1}(A(\lambda)^{-1})B x \\
          \lam^{d_A-2} (A(\lambda)^{-1})Bx \\
          \vdots \\
          (A(\lambda)^{-1})Bx \\
          \hline
          \
              \lam^{d_D-1} x\\
              \lam^{d_D-2} x \\
              \vdots \\
              x
            \end{array}
      \right]. $$

%where
%
%\[
%x := A(\lambda)^{-1}\left[
%\begin{array}{cccc}
%0 & 0 & \cdots & B \\
%\end{array}
%\right]\left[
%\begin{array}{c}
%\lam^{d_A-1}u \\
%\lam^{d_A-2}u \\
%\vdots \\
%u \\
%\end{array}
%\right].
%\]
%
%We refer the pencil $\mathcal{C}_1(\lam)$ given in (\ref{C1}) as the \emph{first companion form of} $\mathcal{S}(\lam)$.

It is easy to see that if $\lam$ is an eigenvalue of $R(\lam)$ then $R(\lam)x = 0$ if and
only if $\mathcal{C}_{1}(\lam)w = 0$

An important class of linearizations (which include the first companion form (\ref{C1}) as special case) that has recently received a lot of attention are the Fielder pencils, \cite{AV04,MFT,TDM}. Introducing \emph{Fiedler matrices}  $M_{i}$, $i= 0, 1, \ldots, d_A$ associated with $A(\lam)= \sum_{i=0}^{d_A}\lam^{i} A_i\in {\mathbb C}[\lambda]^{n,n} $ of degree $d_A$, defined by
\begin{eqnarray}\label{imfp}
M_{d_A} &:=& \left[
          \begin{array}{cc}
            A_{d_A} &  \\
             & I_{(d_A-1)n} \\
          \end{array}
        \right],\ M_{0} := \left[
                   \begin{array}{cc}
                     I_{(d_A-1)n} &  \\
                      & -A_{0} \\
                   \end{array}
                 \right], \nonumber
\\
M_{i} &:=& \left[
  \begin{array}{cccc}
    I_{(d_A-i-1)n} &  & & \\
     & -A_{i} & I_{n} & \\
     & I_{n} & 0  &  \\
     &   &   & I_{(i-1)n}\\
  \end{array}
\right], \,\, i= 1, \ldots, d_A-1,
\end{eqnarray}
see e.g. \cite{TDM}.
%For $i= 1, \ldots, d_A-1$, clearly the $M_{i}$ are invertible
%
%\begin{equation}
%M_{i}^{-1} = \left[
%  \begin{array}{cccc}
%    I_{(d_A-i-1)n} &  & & \\
%     & 0 & I_{n} & \\
%     & I_{n} & A_{i}  &  \\
%     &   &   & I_{(i-1)n} \\
%  \end{array}
%\right], \label{ifp}
%\end{equation}
%
%and for $|i - j| > 1$ the following commutativity relations hold
%
%\begin{eqnarray*}
%M_{i}M_{j} &=& M_{j}M_{i}, \\
%M_{i}^{-1}M_{j} &=& M_{j}M_{i}^{-1},\\
% M_{i}^{-1}M_{j}^{-1} &=& M_{j}^{-1}M_{i}^{-1},
%\end{eqnarray*}
%
%whenever the inverses are defined.

If $\sigma : \{0, 1, \ldots, m-1\}\rightarrow \{ 1, 2, \ldots, m\}$ is a bijection, then one furthermore defines the products $M_{\sigma} := M_{\sigma^{-1}(1)}M_{\sigma^{-1}(2)}\cdots M_{\sigma^{-1}(m)}.$ Note that $\sigma(i)$ describes the position of the factor $M_{i}$ in the product $M_{\sigma}$; i.e., $\sigma(i) = j$ means that $M_{i}$ is the $j$th factor in the product.

Based on the Fiedler matrices,  then for given $A(\lambda)\in {\mathbb F}[\lambda]^{n,n}$ of degree $d_A$ and a bijection $\sigma$, in \cite{TDM} the associated \emph{Fiedler pencil} is defined as the $d_An \times d_An$ matrix pencil
\begin{equation}
L_{\sigma}(\lambda) := \lambda M_{d_A}- M_{\sigma^{-1} (1)}\cdots M_{\sigma^{-1} (d_A)} = \lambda M_{d_A}-  M_{\sigma}. 
\end{equation}
This concept was extended in \cite{rafinami, rafinami2,rafinami3,behera20, behera, beh21}  for square Rosenbrock systems of the state-space form (\ref{tils}) and  (\ref{mulsss}).
In \cite{rafinami} also a multiplication-free algorithm is presented to construct Fiedler pencils for square system polynomials and it is shown that these Fiedler pencils are linearizations of the system polynomial and as well as of the associated transfer functions under some appropriate conditions.

Extending the definition of \cite{rafinami}, based  on the idea of the companion like form (\ref{C1}), we define  $n d_A \times n d_A$ Fiedler matrices associated with $A(\lam)\in {\mathbb F}[\lambda]^{n,n}$ as in (\ref{imfp}), and  Fiedler matrices associated with the matrix polynomial $D(\lam)\in {\mathbb F}[\lambda]^{m,m}$ by
\begin{eqnarray}
N_{d_D} &:=& \left[
          \begin{array}{cc}
            D_{d_D} &  \\
             & I_{(d_D-1)m} \\
          \end{array}
        \right],\ N_{0} := \left[
                   \begin{array}{cc}
                     I_{(d_D-1)m} &  \\
                      & -D_{0} \\
                   \end{array}
                 \right], \nonumber \\
                 \label{0nnfq}
N_{i} &:=& \left[
  \begin{array}{cccc}
    I_{(d_D-i-1)m} &  & & \\
     & -D_{i} & I_{p} & \\
     & I_{m} & 0  &  \\
     &   &   & I_{(i-1)m}\\
  \end{array}
\right], \,\, i= 1, \ldots, d_D-1.
\end{eqnarray}
Based on the Fiedler matrices,  then for given $D(\lambda)\in {\mathbb F}[\lambda]^{m,m}$ of degree $d_D$ and a bijection $\sigma$, in \cite{TDM} the associated \emph{Fiedler pencil} is defined as the $d_Dm \times d_Dm$ matrix pencil
\begin{equation}
T_{\sigma}(\lambda) := \lambda N_{d_D}- N_{\sigma^{-1} (1)}\cdots N_{\sigma^{-1} (d_D)} = \lambda N_{d_D}-  N_{\sigma}. \label{fpp}
\end{equation}
Note that  $M_iM_j=M_jM_i$, $N_i N_j = N_j N_i$ for $|i-j| > 1$ and except for the terms with index $0$, $d_A$ and $d_D$, respectively, each $M_i$ and $N_i$ is invertible.
% with
%
%\begin{equation}
%N_{i}^{-1} = \left[
%  \begin{array}{cccc}
%    I_{(d_A-i-1)m} &  & & \\
%     & 0 & I_{m} & \\
%     & I_{p} & D_{i}  &  \\
%     &   &   & I_{(i-1)m} \\
%  \end{array}
%\right]. \label{ifq}
%\end{equation}
%
We then have the following definition of Fiedler matrices for  Rosenbrock matrices $S(\lambda)\in {\mathbb C}[\lambda]^{n+m,n+m}$ given in (\ref{smpq}).
\begin{definition}\label{mifpq}
Consider a system polynomial $S(\lambda)$ as in (\ref{smpq}) and let $d=\max\{d_A,d_D\}$, $r=\min\{d_A,d_D\}$. Define
$(d_A n+d_D m) \times (d_A n+d_D m)$
matrices $\mathbb{M}_0, \ldots, \mathbb{M}_{d}$ by
\begin{eqnarray*}
\mathbb{M}_0 &=& \left[ {\begin{array}{cc|cc}
I_{(d_A-1)n} & & & \\
 & -A_0 &  & (e_{d_A} e_{d_D}^{T}) \otimes B \\
\hline
 &  & I_{(d_D-1)m} &  \\
 &-(e_{d_D} e_{d_A}^{T})\otimes C  & & -D_0 \\
\end{array}}\right] \\
&=& \left[
                        \begin{array}{c|c}
                          M_0 & (e_{d_A} e_{d_D}^{T}) \otimes B \\
                          \hline
                          -(e_{d_D} e_{d_A}^{T})\otimes C & N_0 \\
                        \end{array}
                      \right] \\
&=& \left[
                        \begin{array}{c|c}
                          M_0 & (e_{d_A} \otimes B)( e_{d_D}^{T} \otimes I_{m}) \\
                          \hline
                          -( e_{d_D}\otimes I_{m})( e_{d_A}^{T} \otimes C) & N_0 \\
                        \end{array}
                      \right], \\
\mathbb{M}_{d} &:=& \left[ {\begin{array}{cc|cc}
 A_{d_A} & & & \\
 & I_{(d_A-1)n} & & \\
\hline
 &  &  D_{d_D} &  \\
 &  & & I_{(d_D-1)m} \\
\end{array}}\right] = \left[
                        \begin{array}{c|c}
                          M_{d_A} &  \\
                          \hline
                           & N_{d_D} \\
                        \end{array}
                      \right], \\
\mathbb{M}_i &:=& \left[
    \begin{array}{cccc|cccc}
      I_{(d_A-i-1)n} &  &  &  &  &  &  &  \\
       & -A_i & I_n &  &  &  &  &  \\
       & I_n & 0 &  &  &  &  &  \\
       &  &  & I_{(i-1)n} &  &  &  &  \\
      \hline
       &  &  &  & I_{(d_D-i-1)m} &  &  &  \\
       &  &  &  &  & -D_i & I_p &  \\
       &  &  &  &  & I_m & 0 &  \\
       &  &  &  &  &  &  & I_{(i-1)m} \\
    \end{array}
  \right] \\
  &=& \left[
              \begin{array}{c|c}
                M_i &  \\
                \hline
                 & N_i \\
              \end{array}
            \right],\ i=1,\ldots,r-1,
\end{eqnarray*}
and
\begin{eqnarray*}
\mathbb{M}_i &:=& \left[
                   \begin{array}{cccc|c}
                     I_{(d_A-i-1)n} &  &  &  &  \\
                      & -A_i & I_n &  &  \\
                      & I_n & 0 &  &  \\
                      &  &  & I_{(i-1)n} &  \\
                     \hline
                      &  &  &  & I_{d_D m} \\
                   \end{array}
                 \right] \\
                 &=& \left[
              \begin{array}{c|c}
                M_i &  \\
                \hline
                 & I_{d_D m} \\
              \end{array}
            \right],\ i= r, r+1,\ldots,d_A-1, \ \mbox{\rm if }\; d_D < d_A,\\
\mathbb{M}_i &:=& \left[
                   \begin{array}{c|cccc}                     I_{d_A n} &  &&&\\
                           \hline
         & I_{(d_D-i-1)m} &  &  &  \\
         &  & -D_i & I_p &  \\
         &  & I_m & 0 &  \\
         &  &  &  & I_{(i-1)m} \\
    \end{array}
  \right] \\
                 &=& \left[
              \begin{array}{c|c}
                I_{d_A n} &  \\
                \hline
                 & N_i\\
              \end{array}
            \right],\ i=d_A, d_A+1,\ldots,d_D-1, \ \mbox{\rm if}\; d_D> d_A.
\end{eqnarray*}
\end{definition}

Observe that as in \cite{rafinami} one has  $\mathbb{M}_i \mathbb{M}_j = \mathbb{M}_j \mathbb{M}_i$  for  $ |i-j| > 1$ and all $\mathbb{M}_i$ (except possibly $\mathbb{M}_0$, $\mathbb M_{d}$) are invertible.

The associated \emph{Fiedler pencils} are then defined as follows.
\begin{definition}
Consider a system polynomial $S(\lambda)$ as in (\ref{smpq}) and let $d=\max\{d_A,d_D\}$, $r=\min\{d_A,d_D\}$.
Let $\mathbb{M}_{0}, \ldots, \mathbb{M}_{d}$ be Fiedler matrices associated with $S(\lam)$ as in Definition~\ref{mifpq}. Given any bijection $\sigma: \{0, 1, \ldots, d-1\} \rightarrow \{1, 2, \ldots, d\}$, the matrix pencil
\begin{equation}
\mathbb{L}_{\sigma}(\lambda) := \lambda \mathbb{M}_{d}- \mathbb{M}_{\sigma^{-1} (1)}\cdots \mathbb{M}_{\sigma^{-1} (m)} =: \lambda \mathbb{M}_{d}- \mathbb{M}_{\sigma}. \label{fps}
\end{equation}
is called the {\em Fiedler pencil} of $S(\lam)$ associated with $\sigma$. We also refer to $\mathbb{L}_{\sigma}(\lambda)$ as a Fiedler pencil of $R(\lam).$
\end{definition}

The companion like form given in (\ref{C1}), then is
\[
\mathcal{C}_1(\lam) = \lam \mathbb{M}_d - \mathbb{M}_{d-1} \cdots \mathbb{M}_1 \mathbb{M}_0
\]
and the associated second companion form of $\mathcal{S}(\lam)$ is
\begin{eqnarray}
\mathcal{C}_2(\lam) &=& \lam \mathbb{M}_d +\mathbb{M}_{0} \mathbb{M}_1 \cdots \mathbb{M}_{d-2} \mathbb{M}_{d-1}\\
&=& \lam \left[
                         \begin{array}{cccc|cccc}
                           A_{d_A} &  &  &  & & & & \\
                            & I_{n} &  &  & & & & \\
                            &  & \ddots &  & & & & \\
                            &  &  & I_{n} & & & & \\
                            \hline
                            &  &  &  & D_{d_D} & & & \\
                            &  &  &  &  & I_m & & \\
                            &  &  &  &  & & \ddots & \\
                            &  &  &  &  & &  & I_m \\
                         \end{array}
                       \right]\nonumber\\
                       && \quad+ \left[
                \begin{array}{cccc|cccc}
                  A_{d_A-1} & -I_{n} &  &  & 0 & & & \\
                  A_{d_A-2} & 0 & \ddots &  & \vdots & & & \\
                  \vdots & \ddots &  & -I_{n} & 0 & & \ddots & \\
                  A_{0} & \cdots & 0 & 0  & -B &  0 & \cdots & 0 \\
                  \hline
                  0  &  &  &   &  D_{d_D} & -I_m & \cdots & 0 \\
                  \vdots  &  \ddots  &  &   &  D_{d_D-1} & & & \vdots\\
                  0  &  &   \ddots &   &  \vdots & & & -I_m\\
                  C  & 0 & \cdots &  0 &  D_0 & 0 & \cdots & 0\\
                \end{array}
              \right].\label{C2}
\end{eqnarray}
%
%In other words,  a bijection  $\sigma$ has a consecution at $j$ if and only if $\mathbb{M}_{j}$ is to the left of $\mathbb{M}_{j+1}$ in %$\mathbb{M}_{\sigma}$, while $\sigma$ has an inversion at $j$ if and only if $\mathbb{M}_{j}$ is to the right of $\mathbb{M}_{j+1}$ in %$\mathbb{M}_{\sigma}$.
%%Further, we have $$M_i M_j = M_j M_i \Longleftrightarrow \mathbb{M}_i \mathbb{M}_j = \mathbb{M}_j \mathbb{M}_i \,\, \mbox{ for } \,\, i, j \in %\sig.$$

\begin{example}
Let $R(\lam) = A(\lam)+ C D(\lam)^{-1}B \in \C^{n \times n}$ with $A(\lam) = A_0+\lam A_1 + \lam^2 A_2 + \lam^3 A_3, A_i \in \C^{n \times n}$ be  a matrix polynomial of degree $3$ and $D(\lam) = D_0+\lam D_1 + \lam^2 D_2, D_i \in \C^{m \times m}$ be a matrix polynomial of degree $2$. Here $d_A > d_D$, $r= 2$ and $d=3$. Let $\sig_1 = (1, 3, 2)$ and $\sig_2 = (2, 3, 1)$ be bijections from $\{0, 1, 2\}$ to $\{1, 2, 3\}$. Then $\mathbb{L}_{\sig_1}(\lam) = \lam \mathbb{M}_3 - \mathbb{M}_0 \mathbb{M}_2 \mathbb{M}_1$ and $\mathbb{L}_{\sig_2}(\lam) = \lam \mathbb{M}_3 - \mathbb{M}_2 \mathbb{M}_0 \mathbb{M}_1$. Then the Fiedler matrices for $G(\lam)$ are given by
$$\mathbb{M}_0 = \left[
                   \begin{array}{ccc|cc}
                    I_n & 0 & 0 & 0 & 0 \\
                     0 & I_n & 0 & 0 & 0 \\
                     0 & 0 & -A_0 & 0 &  B \\
                     \hline
                      0 & 0 & 0 & I_m &  0 \\
                     0 & 0 & -C & 0 & -D_0 \\
                   \end{array}
                 \right] \indent
\mathbb{M}_1 = \left[
                   \begin{array}{ccc|cc}
                    I_n & 0 & 0 &  & \\
                     0 & -A_1 & I_n &  & \\
                     0 & I_n & 0 &  &  \\
                     \hline
                      &  &  & -D_1 & I_m \\
                      &  &  & I_m & 0 \\
                   \end{array}
                 \right]
$$
$$\mathbb{M}_2 = \left[
                   \begin{array}{ccc|cc}
                    -A_2 & I_n & 0 &  & \\
                     I_n & 0 & 0 &  & \\
                     0 & 0 & I_n &  & \\
                     \hline
                      &  &  & I_m & \\
                      &  &  & & I_m \\
                   \end{array}
                 \right]  \indent
\mathbb{M}_3 = \left[
                   \begin{array}{ccc|cc}
                    A_3 &  &  &  & \\
                       & I_n &  &  & \\
                      &  & I_n &  & \\
                     \hline
                      &  &  & D_2 & \\
                      &  &  & & I_m \\
                   \end{array}
                 \right]. $$
Then
$$\mathbb{M}_{\sig_1} = \left[
                   \begin{array}{ccc|cc}
                    -A_2 & -A_1 & I_n & 0 & 0 \\
                     I_n & 0 & 0 & 0 & 0 \\
                     0 &  -A_0 & 0 & B & 0 \\
                     \hline
                     0 & 0 & 0 & -D_1 & I_m \\
                     0 & -C & 0 & -D_0 & 0 \\
                   \end{array}
                 \right].$$
By using the commutativity relation it is easy to check that $\mathbb{L}_{\sig_1}(\lam) = \mathbb{L}_{\sig_2}(\lam). $
\end{example}

\begin{example}
Let $G(\lam) = A(\lam)+ C D(\lam)^{-1}B \in \C^{n \times n}$ with $A(\lambda) = A_0 + \lambda A_1 + \lambda^2A_2 + \lambda^3A_3,  \, A_i \in \mathbb{C}^{n\times n}$ and  $D(\lambda) = D_0 + \lambda D_1 + \lambda^2 D_2 + \lambda^3 D_3 + \lambda^4 D_4, \, D_i \in \mathbb{C}^{m\times m}.$ Here, $d_A = 3$, $d_D=4$, $d_A<d_D$ and $r= 3$,  $d=4$. Consider $\mathbb{L}_{\sig}(\lam) = \lam \mathbb{M}_4 - \mathbb{M}_2 \mathbb{M}_0 \mathbb{M}_1\mathbb{M}_3$. Then the Fiedler matrices for $G(\lam)$ are given by
{\tiny
\[\mathbb{M}_0= \left[\begin{array}{ccc|cccc} I_n & 0 &0&0&0&0&0 \\ 0&I_n &0&0&0&0&0 \\ 0&0&-A_0 & 0&0&0&B \\ \hline 0&0&0& I_m & 0&0&0 \\ 0&0&0&0&I_m & 0&0 \\ 0&0&0&0&0&I_m&0 \\ 0&0&-C&0&0&0&-D_0 \end{array}\right],\ \mathbb{M}_1 = \left[\begin{array}{ccc|cccc}I_n & 0 &0&0&0&0&0 \\ 0&-A_1&I_n&0&0&0&0 \\ 0&I_n&0 & 0&0&0&0 \\ \hline 0&0&0& I_m & 0&0&0 \\ 0&0&0&0&I_m & 0&0 \\ 0&0&0&0&0&-D_1&I_m \\ 0&0&0&0&0&I_m&0\end{array}\right] \]
\[\mathbb{M}_2 = \left[\begin{array}{ccc|cccc}-A_2 &I_n&0&0&0&0&0 \\ I_n&0&0&0&0&0&0 \\ 0&0&I_n& 0&0&0&0 \\ \hline 0&0&0& I_m & 0&0&0 \\ 0&0&0&0&-D_2 &I_m&0 \\ 0&0&0&0&I_m&0&0 \\ 0&0&0&0&0&0&I_m\end{array}\right], \ \mathbb{M}_3 = \left[\begin{array}{ccc|cccc} I_n & 0 &0&0&0&0&0 \\ 0&I_n &0&0&0&0&0 \\ 0&0&I_n & 0&0&0&0 \\ \hline 0&0&0& -D_3 &I_m&0&0 \\ 0&0&0&I_m&0 & 0&0 \\ 0&0&0&0&0&I_m&0 \\ 0&0&0&0&0&0&I_m\end{array}\right]\]
$$\mathbb{M}_4 = \left[\begin{array}{ccc|cccc} A_3 & 0 &0&0&0&0&0 \\ 0&I_n &0&0&0&0&0 \\ 0&0&I_n & 0&0&0&0 \\ \hline 0&0&0& D_4 & 0&0&0 \\ 0&0&0&0&I_m & 0&0 \\ 0&0&0&0&0&I_m&0 \\ 0&0&0&0&0&0&I_m\end{array}\right].$$} Then
$$\mathbb{M}_{\sigma} = \left[\begin{array}{ccc|cccc} -A_2 & -A_1 &I_n&0&0&0&0 \\ I_n&0 &0&0&0&0&0 \\ 0&-A_0& 0& 0&0&B&0 \\ \hline 0&0&0& -D_3 & I_m&0&0 \\ 0&0&0&-D_2&0 & -D_1&I_m \\ 0&0&0&I_m&0&0&0 \\ 0&-C&0&0&0&-D_0&0\end{array}\right]. $$
\end{example}

\begin{example}		
Let $G(\lam) = A(\lam)+ C D(\lam)^{-1}B \in \C^{n \times n}$ with  $A(\lambda) = A_0 + \lambda A_1 + \lambda^2A_2 + \lambda^3A_3$  where $A_i \in \mathbb{C}^{n\times n}$ and  $D(\lambda) = D_0 + \lambda D_1 + \lambda^2 D_2 + \lambda^3 D_3 $ where $D_i \in \mathbb{C}^{m\times m}$. Here, $d_A = 3$, $d_D=3$, $r= 3$ and $d=3$.  Consider $\mathbb{L}_{\sig}(\lam) = \lam \mathbb{M}_3 -\mathbb{M}_{\sigma} = \lam \mathbb{M}_3- \mathbb{M}_2\mathbb{M}_0\mathbb{M}_1$. Then the Fiedler matrices for $G(\lam)$ are given by
{\scriptsize\[\mathbb{M}_0= \left[\begin{array}{ccc|ccc} I_n & 0 &0&0&0&0 \\ 0&I_n &0&0&0&0 \\ 0&0&-A_0 & 0&0&B \\ \hline 0&0&0& I_m & 0&0 \\ 0&0&0&0&I_m & 0 \\ 0&0&C&0&0&-D_0 \end{array}\right],\ \mathbb{M}_1 = \left[\begin{array}{ccc|ccc}I_n & 0 &0&0&0&0 \\ 0&-A_1&I_n&0&0&0 \\ 0&I_n&0 & 0&0&0 \\ \hline 0&0&0& I_m & 0&0 \\ 0&0&0&0&-D_1 & I_m \\ 0&0&0&0&I_m&0\end{array}\right] \]	
	\[\mathbb{M}_2 = \left[\begin{array}{ccc|ccc}-A_2 & I_n&0&0&0&0 \\ I_n & 0 &0&0&0&0 \\ 0&0&I_n&0&0&0 \\ \hline 0&0&0&-D_2&I_m&0 \\ 0&0&0& I_m&0&0 \\ 0&0&0&0&0&I_m\end{array}\right],\ \mathbb{M}_3 = \left[\begin{array}{ccc|ccc}A_3 &0&0&0&0&0 \\ 0&I_n&0&0&0&0 \\ 0&0&I_n & 0&0&0\\ \hline 0&0&0& D_3&0 &0 \\ 0&0&0& 0&I_m&0 \\ 0&0&0&0&0&I_m\end{array} \right]. \]} Then,
	$$\mathbb{M}_{\sigma} = \left[\begin{array}{ccc|ccc}-A_2 & -A_1& I_n &0&0&0 \\ I_n &0&0&0&0&0 \\ 0&-A_0&0&0&B&0 \\ \hline 0&0&0&-D_2 &-D_1 & I_m \\ 0&0&0&I_m &0&0 \\ 0&-C&0&0&-D_0&0\end{array}\right]$$
	\end{example}

Having introduced the basic idea of generating Fiedler pencils for
Rosenbrock system polynomials given in (\ref{smpq}), now we will analyze these constructed pencils.

\begin{theorem}\label{fflptflr}
Let $\mathcal{S}(\lam)$ be given in (\ref{smpq}). Let $d= \max(d_A, d_D)$ and $\sigma :\{0, 1, \ldots, d-1\} \rightarrow \{1, 2, \ldots, d\}$ be a bijection. Let $L_{\sig}(\lam)$, $T_{\sig}(\lam)$ and $\mathbb{L}_{\sig}(\lam)$ be the Fiedler pencils of $A(\lam)$ of degree $d_A$, $D(\lam)$ of degree $d_D$ and $\mathcal{S}(\lam)$, respectively, associated with $\sigma,$ that is, $L_{\sigma}(\lam) := \lam M_{d_A} - M_{\sigma}$, $T_{\sigma}(\lam) := \lam N_{d_D} - N_{\sigma}$ and $\mathbb{L}_{\sigma}(\lam) := \lam \mathbb{M}_{d} - \mathbb{M}_{\sigma}$. If $\sig^{-1} = (\sig_1^{-1}, 0, \sig_2^{-1})$ for some bijections $\sigma_{1}$ and $  \sigma_{2},$ then
$$ \mathbb{L}_{\sigma}(\lam) = \left[
      \begin{array}{c|c}
        L_{\sigma}(\lambda) & -M_{\sigma_{1}}(e_{d_A}e_{d_D}^{T}\otimes B)N_{\sigma_{2}} \\
        \hline
        N_{\sigma_{1}}(e_{d_D}e_{d_A}^{T}\otimes C) M_{\sigma_{2}} & T_{\sigma}(\lambda) \\
      \end{array}
    \right]. $$
Further, if CISS$(\sig) = (c_1, i_1, \ldots, c_l, i_l)$ then
$$\mathbb{L}_{\sig}(\lam) = \left[
                               \begin{array}{c|c}
                               L_{\sig}(\lam) & -e_{d_A} e_{d_D-c_{1}}^{T} \otimes B \\
                               \hline
                               e_{d_D} e_{d_A-c_{1}}^{T} \otimes C & T_{\sigma}(\lambda) \\
                               \end{array}
                               \right], \indent \text{ if } c_1 >0
$$ and $$\mathbb{L}_{\sig}(\lam) = \left[
                               \begin{array}{c|c}
                               L_{\sig}(\lam) & -e_{(d_A-i_{1})}e_{d_D}^{T} \otimes B \\
                               \hline
                               e_{(d_D-i_{1})}e_{d_A}^{T} \otimes C & T_{\sigma}(\lambda) \\
                               \end{array}
                               \right],  \indent \text{ if } c_1 = 0. $$
Thus the map $(\mathrm{Fiedler}(A), \mathrm{Fiedler}(D)) \longrightarrow \mathrm{Fiedler}(\mathcal{S}), $ $ (\lam M_{d_A} -M_{\sig}, \lam N_{d_D} -N_{\sig}) \longmapsto \lam \mathbb{M}_m -\mathbb{M}_{\sig}$ is a bijection, where $\mathrm{Fiedler}(A), \mathrm{Fiedler}(D)$ and $\mathrm{Fiedler}(\mathcal{S})$, respectively, denote the set of Fiedler pencils of $A(\lam)$, $D(\lam)$ and $\mathcal{S}(\lam)$.
\end{theorem}

\begin{proof}
 We have $\mathbb{L}_{\sig}(\lam) = \lam \mathbb{M}_{m} - \mathbb{M}_{\sig} = \lam \mathbb{M}_{m} - \mathbb{M}_{\sig_{1}} \mathbb{M}_{0}\mathbb{M}_{\sig_{2}} $
\begin{align*}  &= {\scriptsize \lam \left[
              \begin{array}{c|c}
               M_{d_A} & 0 \\
                \hline
                0 & N_{d_D} \\
              \end{array}
            \right] - \left[
                        \begin{array}{c|c}
                          M_{\sigma_{1}} & 0  \\
                          \hline
                           0 &  N_{\sigma_{1}}\\
                        \end{array}
                      \right]\left[
                        \begin{array}{c|c}
                          M_0 & (e_{d_A} e_{d_D}^{T}) \otimes B \\
                          \hline
                          -(e_{d_D} e_{d_A}^{T})\otimes C & N_0 \\
                        \end{array}
                      \right]\left[
                                      \begin{array}{c|c}
                                        M_{\sigma_{2}} & 0 \\
                                        \hline
                                        0 & N_{\sigma_{2}} \\
                                      \end{array}
                                    \right]} \\
& =  \lam \left[
              \begin{array}{c|c}
               M_{m} & 0 \\
                \hline
                0 & N_{d_D} \\
              \end{array}
            \right]
- \left[
   \begin{array}{c|c}
   M_{\sigma_{1}} M_{0} M_{\sigma_{2}} & M_{\sigma_{1}}(e_{d_A}e_{d_D}^{T}\otimes B)N_{\sigma_{2}} \\
    \hline
     -N_{\sigma_{1}}(e_{d_D}e_{d_A}^{T}\otimes C) M_{\sigma_{2}}& N_{\sigma_{1}} N_{0} N_{\sigma_{2}}  \\
      \end{array}
      \right] \\
& = \left[
      \begin{array}{c|c}
        L_{\sigma}(\lambda) & -M_{\sigma_{1}}(e_{d_A}e_{d_D}^{T}\otimes B)N_{\sigma_{2}} \\
        \hline
        N_{\sigma_{1}}(e_{d_D}e_{d_A}^{T}\otimes C) M_{\sigma_{2}} & T_{\sigma}(\lambda) \\
      \end{array}
    \right].
\end{align*}
Now suppose that CISS$(\sig) = (c_1, i_1, \ldots, c_l, i_l)$.

Case $I:$ Suppose that $c_1 > 0$. Then by commutativity relation we have $\mathbb{M}_{\sig} = \mathbb{M}_{\sig_1}\mathbb{M}_0 \mathbb{M}_1 \cdots \mathbb{M}_{c_1}$ with $c_1+1 \in \sig_1$. Thus $\mathbb{M}_{\sig} = \mathbb{M}_{\sig_1} \mathbb{M}_0 \mathbb{M}_{\sig_2}$, where $\mathbb{M}_{\sig_2} = \mathbb{M}_1 \cdots \mathbb{M}_{c_1}$. Hence
\begin{align*} \mathbb{M}_{\sig}  &= \left[
       \begin{array}{c|c}
         M_{\sig_1} &  \\
         \hline
          & N_{\sig_1} \\
       \end{array}
     \right]\left[
                        \begin{array}{c|c}
                          M_0 & (e_{d_A} e_{d_D}^{T}) \otimes B \\
                          \hline
                          -(e_{d_D} e_{d_A}^{T})\otimes C & N_0 \\
                        \end{array}
                      \right]\left[
                                      \begin{array}{c|c}
                                        M_{\sigma_{2}} & 0 \\
                                        \hline
                                        0 & N_{\sigma_{2}} \\
                                      \end{array}
                                    \right]\\
 & = \left[
   \begin{array}{c|c}
   M_{\sigma_{1}} M_{0} M_{\sigma_{2}} & M_{\sigma_{1}}(e_{d_A}e_{d_D}^{T}\otimes B)N_{\sigma_{2}} \\
    \hline
     -N_{\sigma_{1}}(e_{d_D}e_{d_A}^{T}\otimes C) M_{\sigma_{2}}& N_{\sigma_{1}} N_{0} N_{\sigma_{2}}  \\
      \end{array}
      \right].
\end{align*}
Since $j \in \sig_1$ implies that $j \geq c_1+1$, we have $M_{\sig_1} = \left[
                                                                                  \begin{array}{c|c}
                                                                                  * &  \\
                                                                                  \hline
                                                                                  & I_{c_1n} \\
                                                                                  \end{array}
                                                                                  \right]
$ and $N_{\sig_1} = \left[
                                                                                  \begin{array}{c|c}
                                                                                  * &  \\
                                                                                  \hline
                                                                                  & I_{c_1m} \\
                                                                                  \end{array}
                                                                                  \right]
$. This shows that $M_{\sig_1} (e_{d_{A}} \otimes I_n) = e_{d_A} \otimes I_n$ and $N_{\sig_1} (e_{d_{D}} \otimes I_m) = e_{d_D} \otimes I_m$. So, we have
$ M_{\sig_1} (e_{d_A} \otimes B) = e_{d_A} \otimes B$. Next, we have $N_{\sigma_{1}}(e_{d_D}e_{d_A}^{T}\otimes C) M_{\sigma_{2}} = N_{\sigma_{1}}(e_{d_D} \otimes I_m)(e_{d_A}^{T}\otimes C) M_{\sigma_{2}} $ and $M_{\sigma_{1}}(e_{d_A}e_{d_D}^{T}\otimes B)N_{\sigma_{2}} = M_{\sigma_{1}}(e_{d_A} \otimes B )( e_{d_D}^{T}\otimes I_{m})N_{\sigma_{2}}$. Now, we have
%for $i=1:r$, where $r=\min(d_A, d_D)$ we have
\begin{align*} ( e_{d_A}^{T} \otimes I_n) M_1 &= (e_{d_A}^{T} \otimes I_n) \left[
                                                               \begin{array}{ccc}
                                                                 I_{(d_{A}-2)n} &  & \\
                                                                  & -A_1 & I_n  \\
                                                                  & I_n & 0  \\
                                                               \end{array}
                                                             \right]
                                                = (e_{d_A-1}^{T} \otimes I_n),\\
(e_{d_A}^{T} \otimes I_n) M_1 M_2 &= (e_{d_A-1}^{T} \otimes I_n) \left[
                                                               \begin{array}{cccc}
                                                                 I_{(d_{A}-3)n} &  &  &  \\
                                                                  & -A_2 & I_n &   \\
                                                                  & I_n & 0 &    \\
                                                                  &  &  & I_n   \\
                                                               \end{array}
                                                             \right]
                                                = (e_{d_A-2}^{T} \otimes I_n),\\
 \end{align*} and so on. Thus
$(e_{d_A}^{T} \otimes I_n) M_1 M_2 \cdots M_{c_1} = (e_{d_A-c_1}^{T} \otimes I_n)$.
Hence $(e_{d_A}^{T} \otimes I_n)M_{\sig_2} = (e_{d_A-c_1}^{T} \otimes I_n)$ and $(-(e_{d_A-c_1}^{T} \otimes C)M_{\sig_2} = -(e_{d_A-c_1}^{T} \otimes C)$.
Similarly, we have
\begin{align*} ( e_{d_D}^{T} \otimes I_m) N_1 &= (e_{d_D}^{T} \otimes I_m) \left[
                                                               \begin{array}{ccc}
                                                                     I_{(d_{D}-2)m} &  \\
                                                                    & -D_1 & I_m  \\
                                                                   & I_m & 0 \\
                                                               \end{array}
                                                             \right]
                                                = (e_{d_D-1}^{T} \otimes I_m),\\
(e_{d_D}^{T} \otimes I_m) N_1 N_2 &= (e_{d_D-1}^{T} \otimes I_m) \left[
                                                               \begin{array}{cccc}
                                                                     I_{(d_{D}-3)m} & & & \\
                                                                    & -D_2  & I_m  & \\
                                                                    &  I_m & 0 &  \\
                                                                    & &  & I_m  \\
                                                               \end{array}
                                                             \right]
                                                = (e_{d_D-2}^{T} \otimes I_m),\\
 \end{align*} and so on. Thus
$(e_{d_D}^{T} \otimes I_m) N_1 N_2 \cdots N_{c_1} = (e_{d_D-c_1}^{T} \otimes I_m)$.
Hence $(e_{d_D}^{T} \otimes I_m)N_{\sig_2} = (e_{d_D-c_1}^{T} \otimes I_m)$.
Now, we have $N_{\sigma_{1}}(e_{d_D}e_{d_A}^{T} \otimes I_n)M_{\sig_2} = N_{\sigma_{1}}(e_{d_D} \otimes I_m)(e_{d_A}^{T}\otimes C) M_{\sigma_{2}} = (e_{d_D}e_{d_A-c_1}^{T} \otimes I_n)$ and $-(N_{\sigma_{1}}e_{d_D}e_{d_A}^{T} \otimes C)M_{\sig_2} = -(e_{d_D}e_{d_A-c_1}^{T} \otimes C)$.
Similarly, $M_{\sigma_{1}}(e_{d_A}e_{d_D}^{T}\otimes B)N_{\sigma_{2}} =  (e_{d_A}e_{d_D-c_1}^{T} \otimes B)$. Consequently, we have
$$\mathbb{L}_{\sig}(\lam) = \lam \mathbb{M}_{m} - \mathbb{M}_{\sig} = \left[
                               \begin{array}{c|c}
                               L_{\sig}(\lam) & -e_{d_A} e_{d_D-c_{1}}^{T} \otimes B \\
                               \hline
                               e_{d_D} e_{d_A-c_{1}}^{T} \otimes C & T_{\sigma}(\lambda) \\
                               \end{array}
                               \right].$$

Case $II:$ Suppose that $c_1 = 0$. Then $\sig$ has $i_1$ inversions at $0$. Hence by commutativity relations we have $\mathbb{M}_{\sig} = \mathbb{M}_{i_1} \cdots \mathbb{M}_1 \mathbb{M}_0 \mathbb{M}_{\sig_2} =:\mathbb{M}_{\sig_1}\mathbb{M}_0 M_{\sig_2} $ with $i_1+1 \in \sig_2$.  Hence
\begin{align*} \mathbb{M}_{\sig}  &= \left[
       \begin{array}{c|c}
         M_{\sig_1} &  \\
         \hline
          & N_{\sig_1} \\
       \end{array}
     \right]\left[
                        \begin{array}{c|c}
                          M_0 & (e_{d_A} e_{d_D}^{T}) \otimes B \\
                          \hline
                          -(e_{d_D} e_{d_A}^{T})\otimes C & N_0 \\
                        \end{array}
                      \right]\left[
                                      \begin{array}{c|c}
                                        M_{\sigma_{2}} & 0 \\
                                        \hline
                                        0 & N_{\sigma_{2}} \\
                                      \end{array}
                                    \right]\\
 & = \left[
   \begin{array}{c|c}
   M_{\sigma_{1}} M_{0} M_{\sigma_{2}} & M_{\sigma_{1}}(e_{d_A}e_{d_D}^{T}\otimes B)N_{\sigma_{2}} \\
    \hline
     -N_{\sigma_{1}}(e_{d_D}e_{d_A}^{T}\otimes C) M_{\sigma_{2}}& N_{\sigma_{1}} N_{0} N_{\sigma_{2}}  \\
      \end{array}
      \right].\end{align*}
Since $j \in \sig_2$ implies that $j \geq i_1+1$, we have $M_{\sig_2} = \left[
                         \begin{array}{c|c}
                         * &  \\
                         \hline
                         & I_{i_1n} \\
                         \end{array}
                         \right]
$ and $N_{\sig_2} = \left[
                         \begin{array}{c|c}
                         * &  \\
                         \hline
                         & I_{i_1m} \\
                         \end{array}
                         \right]
$. This shows that $ (e_{d_A}^{T} \otimes I_n)M_{\sig_2} = e_{d_A}^{T} \otimes I_n$ and $ (e_{d_D}^{T} \otimes I_m)N_{\sig_2} = e_{d_D}^{T} \otimes I_m$. Hence $ (-e_{d_A}^{T} \otimes C)M_{\sig_2} = -e_{d_A}^{T} \otimes C$. Next, we have
\begin{align*}  M_1 (e_{d_A} \otimes I_n) &= \left[
                                     \begin{array}{ccc}
                                     I_{(m-2)n} &  &  \\
                                     & -A_1 & I_n \\
                                     & I_n & 0 \\
                                     \end{array}
                                     \right]  (e_{d_A} \otimes I_n)= (e_{{d_A}-1} \otimes I_n),\\
M_2 M_1 (e_{d_A} \otimes I_n) &= \left[
                                                               \begin{array}{cccc}
                                                                 I_{({d_A}-3)n} &  & & \\
                                                                  & -A_2 & I_n & \\
                                                                  & I_n & 0 & \\
                                                                   &  &  & I_n \\
                                                               \end{array}
                                                             \right] (e_{{d_A}-1} \otimes I_n)
                                                             = (e_{{d_A}-2} \otimes I_n). \end{align*}  Thus
$ M_{i_1}  \cdots M_2 M_1 (e_{d_{A}} \otimes I_n)= (e_{d_{A}-i_1} \otimes I_n)$. Hence $M_{\sig_1}(e_{d_{A}} \otimes I_n) = (e_{(d_{A}-i_1)} \otimes I_n)$ and $M_{\sig_1}(e_{d_{A}} \otimes B) = (e_{(d_{A}-i_1)} \otimes B)$.
Similarly, we have
\begin{align*}  N_1 (e_{d_D} \otimes I_m) &= \left[
                                     \begin{array}{ccc}
                                     I_{(d_D-2)m} &  &  \\
                                     & -D_1 & I_m \\
                                     & I_m & 0 \\
                                     \end{array}
                                     \right]  (e_{d_D}\otimes I_m)= (e_{{d_D}-1} \otimes I_m),\\
N_2 N_1 (e_{d_D} \otimes I_m) &= \left[
                                                               \begin{array}{cccc}
                                                                 I_{(d_D-3)m} &  & & \\
                                                                  & -D_2 & I_m & \\
                                                                  & I_m & 0 & \\
                                                                   &  &  & I_m \\
                                                               \end{array}
                                                             \right] (e_{d_D-1} \otimes I_m)
                                                             = (e_{d_D-2} \otimes I_m).
                                                             \end{align*}
Thus
$ N_{i_1}  \cdots N_2 N_1 (e_{d_{D}} \otimes I_m)= (e_{d_{D}-i_1} \otimes I_m)$. Hence $N_{\sig_1}(e_{d_{D}} \otimes I_m) = (e_{(d_{D}-i_1)} \otimes I_m)$.
Now, we have $N_{\sigma_{1}}(e_{d_D}e_{d_A}^{T} \otimes I_n)M_{\sig_2} = N_{\sigma_{1}}(e_{d_D} \otimes I_m)(e_{d_A}^{T}\otimes C) M_{\sigma_{2}} = (e_{d_D - i_1}e_{d_A}^{T} \otimes I_n)$ and $-(N_{\sigma_{1}}e_{d_D}e_{d_A}^{T} \otimes C)M_{\sig_2} = -(e_{d_D- i_1}e_{d_A}^{T} \otimes C)$.
Similarly, $M_{\sigma_{1}}(e_{d_A}e_{d_D}^{T}\otimes B)N_{\sigma_{2}} = (e_{d_A-i_1}e_{d_D}^{T} \otimes B)$.
Consequently, we have
$$\mathbb{L}_{\sig}(\lam) = \lam \mathbb{M}_{m} - \mathbb{M}_{\sig} = \left[
                               \begin{array}{c|c}
                               L_{\sig}(\lam) & -e_{(d_A-i_{1})}e_{d_D}^{T} \otimes B \\
                               \hline
                               e_{(d_D-i_{1})}e_{d_A}^{T} \otimes C & T_{\sigma}(\lambda) \\
                               \end{array}
                               \right]. $$
Note that for each $i, j \in \sig$, we have $M_i M_j = M_j M_i, \, N_i N_j = N_j N_i \Leftrightarrow \mathbb{M}_i \mathbb{M}_j = \mathbb{M}_j \mathbb{M}_i$. Hence it follows that $\# (\mathrm{Fiedler}(A), \mathrm{Fiedler}(D)) = \#(\mathrm{Fiedler}(\mathcal{S}))$. This completes the proof.
\end{proof}

\begin{theorem}\label{cowi}
Let $\mathcal{S}(\lam)$ be in (\ref{smpq}) with $A(\lam) = \sum\limits_{i=0}^{d_{A}}\lam^{i}A_i, A_i \in \C^{n\times n} $, $\sum \limits_{i=0}^{d_D}\lam^{i}D_{i}, D_i \in \C^{m\times m}$. Suppose that $d_A > d_D. $ Let $\sigma: \{0, 1, \ldots, d_A-1\} \rightarrow \{1, 2, \ldots, d_A\}$ be a bijection. The following algorithm constructs a sequence of matrices $\{\mathbb{W}_0, \mathbb{W}_1, \ldots, \mathbb{W}_{d_A-2}\}, $ where each matrix $\mathbb{W}_i$ for $i = 1, 2, \ldots, d_A-2$ is partitioned into blocks in such a way that the blocks of $\mathbb{W}_{i-1}$ are blocks of $W_i. $
\begin{algorithm}[H]
\caption{Construction of $\mathbb{M}_{\sigma}$ for $\mathbb{L}_{\sigma}(\lam) := \lam \mathbb{M}_{d_A}- \mathbb{M}_{\sigma}$.}
\label{alg1}

\textbf{Input}: {\scriptsize$\mathcal{S}(\lam) = \left[
                              \begin{array}{c|c}
                                \sum \limits_{i=0}^{d_A}\lam^{i}A_{i} & -B \\
                                \hline
                                C & \sum \limits_{i=0}^{d_D}\lam^{i}D_{i}\\
                              \end{array}
                            \right]$} and a bijection $\sigma :\{0, 1, \ldots, d_A-1\} \rar \{1, 2, \ldots, d_A\} $. \\
\textbf{Output}: { $\mathbb{M}_{\sigma}$ }
\begin{algorithmic}

\If{$\sigma$ has a consecution at $0$}
    \State {\scriptsize$\mathbb{W}_0 := \left[
                    \begin{array}{cc|cc}
                      -A_{1} & I_n & 0 & 0 \\
                      -A_0 & 0 & B & 0 \\
                      \hline
                      0 & 0 & -D_1 & I_m \\
                      -C & 0 & -D_0 & 0 \\
                    \end{array}
                  \right]$}

\Else
    \State {\scriptsize$\mathbb{W}_0 := \left[
                    \begin{array}{cc|cc}
                      -A_{1} & -A_0 & 0 & B \\
                       I_n & 0 & 0 & 0 \\
                      \hline
                      0 & -C & -D_1 & -D_0 \\
                      0 & 0 & I_m & 0 \\
                    \end{array}
                  \right]$}

\EndIf \\
If $d_{A} > d_{D}$  \\
   %\For{$i = 1: r-2$}

\For{$i = 1:d_{D}-2$}

\If{$\sigma$ has a consecution at $i$}
    \State {\scriptsize $\mathbb{W}_i :=  \left[
       \begin{array}{ccc|c}
        -A_{i+1} & I_{n} &  0  & \\
        \mathbb{W}_{i-1}(1:i+1,1) & 0 & \mathbb{W}_{i-1}(1:i+1,2:i+1) & W_{i2}  \\
         \hline
        	0 & 0 &  0 & \\
        \mathbb{W}_{i-1}(3:i+3, 1) & 0 & \mathbb{W}_{i-1}(3:i+3,2:i+1)    & W_{i4}  \\
       \end{array}
     \right]$},   where  \\
  {\scriptsize$W_{i2}= \left[
     \begin{array}{ccc}
     	0 & 0 &  0 \\
     	\mathbb{W}_{i-1}(1:i+1, i+2) & 0 & \mathbb{W}_{i-1}(1:i+1, i+3:2i+2)  \\
     \end{array}
     \right]$}, \\
  {\scriptsize $W_{i4}= \left[
  \begin{array}{ccc}
  -D_{i+1} & I_{m} &  0     \\
  	\mathbb{W}_{i-1}(3:i+3,i+2) & 0 & \mathbb{W}_{i-1}(3:i+3, i+3:2i+2)  \\
  \end{array}
  \right]$.}

\Else
    \State {\scriptsize$\mathbb{W}_i := \left[
           \begin{array}{cc|cc}
             -A_{i+1} & \mathbb{W}_{i-1}(1,1:i+1) & 0 & \mathbb{W}_{i-1}(1,3:i+3) \\
             I_{n} & 0  & 0 &  0 \\
             0 & \mathbb{W}_{i-1}(2:i+1,1:i+1) & 0 & \mathbb{W}_{i-1}(2:i+1,3:i+3) \\
             \hline
             0 & \mathbb{W}_{i-1}(i+2,1:i+1) & -D_{i+1} & \mathbb{W}_{i-1}(i+2,3:i+3) \\
             0 &  0 & I_{m} &   0    \\
             0 & \mathbb{W}_{i-1}(i+3:2i+2,1:i+1) & 0  & \mathbb{W}_{i-1}( i+3:2i+2,3:i+3)  \\
           \end{array}
         \right] $}
\EndIf

\EndFor

\For{$i = d_{D}-1: d_{A}-2$}

\If{$\sigma$ has a consecution at $i$}
    \State{\scriptsize $\mathbb{W}_i :=  \left[
       \begin{array}{cccc}
         -A_{i+1} & I_{n} & 0 & 0  \\
         \mathbb{W}_{i-1}(:,1) & 0 & \mathbb{W}_{i-1}(:,2:i+1) & \mathbb{W}_{i-1}(:,i+2:2i+2)  \\
       \end{array}
     \right]$}

\Else
    \State {\scriptsize$\mathbb{W}_i := \left[
           \begin{array}{cc}
             -A_{i+1} & \mathbb{W}_{i-1}(1,:)   \\
             I_{n} & 0    \\
             0 & \mathbb{W}_{i-1}(2:i+1, :)  \\
             0 & \mathbb{W}_{i-1}(i+2:2i+2, :)   \\
            % 0 & \mathbb{W}_{i-1}(i+3, :) \\
           \end{array}
         \right] $}
%  W_{i-1}(3, 1:i+1) & W_{i-1}(3, 3:i+3)   \\
 %0 & W_{i-1}(i+3, :) \\
\EndIf

\EndFor

\State $\mathbb{M}_{\sigma} := \mathbb{W}_{d_{A}-2}$

\end{algorithmic}
\end{algorithm}
\end{theorem}

\begin{proof}
We prove the result by induction on the degree $d_A= \max(d_A, d_D)$.  Then the rest of the proof follows directly from proof of Theorem~3.11 in \cite{rafinami}.
\end{proof}

\begin{theorem}
Let $\mathcal{S}(\lam)$ be in (\ref{smpq}) with $A(\lam) = \sum\limits_{i=0}^{d_{A}}\lam^{i}A_i, A_i \in \C^{n\times n} $, $\sum \limits_{i=0}^{d_D}\lam^{i}D_{i}, D_i \in \C^{m\times m}$. Suppose that $d_A < d_D. $ Let $\sigma: \{0, 1, \ldots, d_D-1\} \rightarrow \{1, 2, \ldots, d_D\}$ be a bijection. The following algorithm constructs a sequence of matrices $\{\mathbb{W}_0, \mathbb{W}_1, \ldots, \mathbb{W}_{d_D-2}\}, $ where each matrix $\mathbb{W}_i$ for $i = 1, 2, \ldots, d_D-2$ is partitioned into blocks in such a way that the blocks of $\mathbb{W}_{i-1}$ are blocks of $W_i. $
\begin{algorithm}[H]
\caption{Construction of $\mathbb{M}_{\sigma}$ for $\mathbb{L}_{\sigma}(\lam) := \lam \mathbb{M}_{d_D}- \mathbb{M}_{\sigma}$.}
%\textbf{Algorithm 2} Construction of $\mathbb{M}_{\sigma}$ for $\mathbb{L}_{\sigma}=\lambda \mathbb{M}_{d_D}-\mathbb{M}_{\sigma}$

\textbf{Input:} $\mathcal{S}(\lambda) = \left[\begin{array}{c|c} \sum_{i=0}^{d_A}\lambda^i A_i & -B\\ \hline C & \sum_{i=0}^{d_D}\lambda^i D_i \end{array}\right]$ and a bijection $\sigma : \{0,1,\ldots ,d_D-1\}\to \{1,2,\ldots,d_D\}$

\textbf{Output:} $\mathbb{M}_{\sigma}$

\begin{algorithmic}

\If{$\sigma$ has a consecution at $0$ }
\State \scriptsize $\mathbb{W}_0=\left[\begin{array}{cc|cc}
	-A_1&I_n&0 &0\\ -A_0&0&B&0 \\ \hline 0&0&-D_1 &I_m \\ -C & 0&I_m & 0
\end{array}\right]$

\Else
\State \scriptsize$\mathbb{W}_0 =\left[ \begin{array}{cc|cc} -A_1& A_0 & 0 & B \\ I_n & 0 &0 &0 \\ \hline 0&-C &-D_1 &-D_0 \\ 0&0&I_m &0 \end{array}\right]$	

\EndIf

If $d_D>d_A$ \\
\For{ $i = 1:d_A-2$ }
\If {$\sigma$ has a consecution at $i$ }
\State {\scriptsize
$\mathbb{W}_i := \left[\begin{array}{ccc|ccc}-A_{i+1}&I_n&\\ \mathbb{W}_{i-1}(1:i+1,1)&0&\mathbb{W}_{i-1}(1:i+1,2:i+1)&W_{i2}\\ \hline 0&0&0& \\ \mathbb{W}_{i-1}(3:i+3,1)&0&\mathbb{W}_{i-1}(3:i+3,2:i+1)&W_{i4}\end{array}\right], \text{where}$}

\scriptsize$W_{i2}=\left[\begin{array}{ccc}0&0&0\\ \mathbb{W}_{i-1}(1:i+1,i+2)&0&\mathbb{W}_{i-1}(1:i+1,i+3:2i+2)\end{array}\right]$

\scriptsize{$W_{i4}=\left[\begin{array}{ccc}-D_{i+1} I_m&0&0 \\\mathbb{W}_{i-1}(3:i+3,i+2)&0& \mathbb{W}_{i-1}(3:i+3,i+3:2i+2) \end{array}\right]$}
\Else
\State {\scriptsize$\mathbb{W}_i =\left[\begin{array}{cc|cc} -A_{i+1}& \mathbb{W}_{i-1}(1,1:i+1)& 0&\mathbb{W}_{i-1}(1,3:i+3)\\I_n&0&0&0\\ 0& \mathbb{W}_{i-1}(2:i+1,1:i+1)&0&\mathbb{W}_{i-1}(2:i+1,3:i+3)\\ \hline 0&\mathbb{W}_{i-1}(i+2,1:i+1)&-D_{i+1}&\mathbb{W}_{i-1}(i+2,3:i+3)\\0&0&I_m&0\\0&\mathbb{W}_{i-1}(i+3:2i+2,1:i+1)&0&\mathbb{W}_{i-1}(i+3:2i+2,3:i+3)
\end{array}\right]$}

\EndIf
\EndFor

\For {$i=d_A -1:d_D -2 $ }
\If {$\sigma$ has a consecution at $i$}
\State {\tiny
$\mathbb{W}_i = \left[\begin{array}{c|ccc}\mathbb{W}_{i-1}(1:i+1,1:i+1)&0&0&\mathbb{W}_{i-1}(1:i+1,2i+1:2i+2)\\ \hline 0&-D_{i+1}&I_m&0\\\mathbb{W}_{i-1}(i+2:2i+2,i:i+1)&\mathbb{W}_{i-1}(i+2:2i+2,i+2)&0&\mathbb{W}_{i-1}(i+2:2i+2,2i+1:2i+2)\end{array}\right]$}

\Else
\State {\scriptsize$\mathbb{W}_i = \left[\begin{array}{c|cc}\mathbb{W}_{i-1}(1:i+1,1:i+1)&0&\mathbb{W}_{i-1}(1:i+1,i+2:2i+2) \\ \hline 0&-D_{i+1}&\mathbb{W}_{i-1}(i+2,i+2:2i+2)\\0&I_m&0\\\mathbb{W}_{i-1}(2i+1:2i+2,1:i+1) &0&\mathbb{W}_{i-1}(2i+1:2i+2,i+2:2i+2)\end{array}\right]$} \\
\EndIf

\EndFor

\State $\mathbb{M}_{\sigma} := \mathbb{W}_{d_{D}-2}$
 \end{algorithmic}
\end{algorithm}
\end{theorem}

\begin{proof}
We prove the result by induction on the degree $d_D= \max(d_A, d_D)$.  Then the rest of the proof follows directly from proof of Theorem~3.11 in \cite{rafinami}.
\end{proof}

\section{Fiedler linearizations of Rosenbrock system matrix}\label{sec:system}
%\begin{definition}[System polynomial]
%An $(n+r)\times (n+r)$ matrix polynomial $ \mathcal{X}(\lam)$ is said to be a  system polynomial of degree $m$ if it is of the form $$ %\mathcal{X}(\lam) = \left[\begin{array}{c|c} X_{11}(\lam) &  X_{12}\\ \hline  X_{21} &  X_{22}(\lam) \end{array}\right],$$  where $X_{11}(\lam)  %$ is an $n$-by-$n$ matrix polynomial of degree $m$, $X_{12}$ and $ X_{21}$ are constant matrices and $ X_{22}(\lam)$ is an $r$-by-$r$ matrix %polynomial of degree $k <m$.. If, in addition, the matrix polynomial $X_{22}(\lam)$ is regular then $\mathcal{X}(\lam)$ is said to be a %Rosenbrock system polynomial and, in such a case, the rational matrix function $$ T(\lam) := X_{11}(\lam) - X_{12} %\left(X_{22}(\lam)\right)^{-1} X_{21}$$ is said to be the transfer function associated with $\mathcal{X}(\lam).$
%\end{definition}
In this section we show that the constructed Fiedler pencils associated with Rosenbrock systems are indeed linearizations.
To do this we have to recall a few basic facts.
\begin{definition}[System equivalence] Let $ S_1(\lam)$ and $S_2(\lam)$ be $(n+m)\times (n+m)$ Rosenbrock system polynomials of the form (\ref{rosmatrix}), partitioned conformably.  Then $ S_1(\lambda)$ is said to be \emph{system equivalent} to  $S_2(\lambda)$ (denoted as $ S_1(\lam) \thicksim_{se} S_2(\lam)$), if there exist unimodular matrix polynomials $ U(\lambda), V(\lambda)\in \mathbb F^{n,n}$, $ \widetilde{U}(\lam)\in \mathbb F^{m,m}$, and $\widetilde{V}(\lam)\in \mathbb F^{m,m}$ such that  for all   $\lam \in \C$ we have
\begin{equation}\label{sysequi}
 \left[\begin{array}{c|c} U(\lam) &  0 \\  \hline 0 & \widetilde{U}(\lam) \end{array} \right] S_1(\lam)  \left[ \begin{array}{c|c} V(\lam) & 0 \\ \hline  0&  \widetilde{V}(\lam) \end{array}\right] = S_2 (\lam).
\end{equation}
\end{definition}
\begin{definition}[Rosenbrock linearization] Let $ S(\lam)$ be an $ (n+m)\times (n+m)$ system polynomial of the form (\ref{rosmatrix})
with degree $d=\max\{d_A,d_D\}$. A linear matrix polynomial ${L}(\lambda)$ is called a \emph{Rosenbrock linearization} of $S(\lambda)$,
if it has the form
\[
{L}(\lambda):=
\left[
                                                     \begin{array}{c|c}
                                                       {\mathcal A}(\lam) & {\mathcal B}\\
                                                       \hline
                                                       {\mathcal C} & {\mathcal D}(\lam) \\
                                                     \end{array}
                                                   \right],
\]
with matrix polynomials ${\mathcal A}(\lambda),{\mathcal D}(\lambda)$ of degree less than or equal to $1$, (constant in $\lambda$) matrices ${\mathcal B}, {\mathcal C}$,  and ${\mathcal S}(\lambda)$ is system equivalent to
\begin{equation}\label{sysL}
\tilde {\mathcal S}(\lambda):=
 \left[\begin{array}{c|c} U(\lam) &  0 \\  \hline 0 & \widetilde{U}(\lam) \end{array} \right] L(\lam)  \left[ \begin{array}{c|c} V(\lam) & 0 \\ \hline  0&  \widetilde{V}(\lam) \end{array}\right] =\left[
                                           \begin{array}{c|c}
                                             I_{(d-1)n} & 0 \\
                                             \hline
                                             0 &  S(\lam) \\
                                           \end{array}
                                         \right].
\end{equation}
If, in addition, $U(\lam)$ and $V(\lam)$ in  (\ref{sysL}) are constant matrices, then ${\mathcal S}(\lam)$ is said to be a \emph{strict Rosenbrock linearization} of $S(\lam)$.
%An S-linearization  of a Rosenbrock system polynomial $ \mathcal{S}(\lam)$ is said to be  a {\em Rosenbrock linearization} of $ %\mathcal{S}(\lam).$
\end{definition}

Let $E := (E_{ij})$ be a block $m \times n$ matrix with $p\times q$  blocks $E_{ij}$. The {\em block transpose} of $E$, see \cite{TDM}, denoted by $E^{\mathcal{B}},$ is the block $n \times m$ matrix  with $p \times q$ blocks defined by $(E^{\mathcal{B}})_{ij} := E_{ji}$.
%
%\mpar{define Kron. prod.}
We slightly modify this definition for the special structure of Rosenbrock linearizations.

\begin{definition}[Rosenbrock block transpose]\label{btfrmf} Consider a Rosenbrock system matrix of the from (\ref{rosmatrix}) and let $\mathcal S $ be an $(d_An+m)\times (d_An+m)$ Rosenbrock linearization of the form (\ref{sysL}), where ${\mathcal B}:=-(e_i e_j^{T}) \otimes B$ and
${\mathcal C}:=(e_k e_{\ell}^{T}) \otimes C$ with $B\in \mathbb F^{n,m}$, $C\in \mathbb F^{p,n}$, where
${\mathcal A} := [{\mathcal A}_{ij}]$ is an $d_A \times d_A$ block matrix with ${\mathcal A}_{ij} \in {\mathbb F}^{n \times n}$,
and where ${\mathcal D}$ is a $d_D \times d_D$ block matrix with $D_{ij} \in {\mathbb F}^{p \times m}$, or ${\mathbb F}^{m \times m}$ or ${\mathbb F}^{p,p}$. The {\em Rosenbrock block transpose} of $\mathcal{S}$, denoted by $\mathcal{S}^{\mathbb{B}}$ is defined by
\[
\mathcal{S}^{\mathbb{B}} := \left[
                 \begin{array}{c|c}
                   A(\lambda)^{\mathcal{B}} & -(e_{\ell} e_k^{T}) \otimes B \\
                   \hline
                   (e_j e_i^{T})  \otimes C & D(\lambda)^{\mathcal{B}} \\
                 \end{array}
               \right].
               \]
\end{definition}
For $\mathcal{C}_1(\lam)$ and $\mathcal{C}_2(\lam)$ given in (\ref{C1}) and (\ref{C2}), respectively, we have $\mathcal{C}_2(\lam) = \mathcal{C}_1(\lam)^{\mathbb{B}}$.

%We also extend the concept of Horner shifts, \cite{DeTDM09}. Let $P(\lambda) = P_{0} + \lambda P_{1} + \cdots + \lambda^{d_P}P_{d_P}\in {\mathbb F}[\lambda]^{p,m}$  be of degree $d_P$. For $k = 0, \ldots d_P$, the degree $k$ \emph{Horner shift} of $P(\lambda)$ is the matrix polynomial $\mathcal P_{k}(\lambda) := P_{d_P-k} + \lambda P_{d_P-k+1} + \cdots + \lambda^{k}P_{d_P}$.
%Horner shifts have the following obvious properties:
%
%\begin{align}
%& {\mathcal P}_{0}(\lambda) = P_{d_P}, \nonumber \\
%& {\mathcal P}_{k+1}(\lambda) = \lambda {\mathcal P}_{k}(\lambda) + P_{d_P-k-1} , \mbox{  for } 0\leq k\leq d_P-1, \nonumber\\
%& {\mathcal P}_{d_P}(\lambda) = {\mathcal P}(\lambda).  \label{hs}
%\end{align}
%
Extending [\cite{TDM}, Definition~4.2] we define auxiliary matrix polynomials associated with Horner shifts for system polynomials.
%

%{\bf From here on needs a major revision. We should right away go for rectangular system polynomials with  square A and rect. D and we should also go for general D.
%Can you please rewrite this part in line with the part that I have rewritten.
%If you wish we can skype over this next week.
%
%
%Best volker}

%
\begin{definition} \label{qrtdmfmp}
Let $A(\lambda) = \sum\limits_{i = 0}^{d_A}\lambda^{i} A_{i}\in {\mathbb C}[\lambda]^{n,n}$ be of degree $d_A$ and let $ P_{i}(\lambda)$ be the degree $i$ Horner shift of $A(\lambda)$. For $1 \leq i \leq d_A-1$, define the matrix polynomials
$$
 Q_{i}(\lambda) := \left[
                    \begin{array}{cccc}
                      I_{(i-1)n} &  &  &  \\
                       & I_{n} & \lambda I_{n} &  \\
                       & 0_{n} & I_{n} &  \\
                       &  &  & I_{(d_A-i-1)n} \\
                    \end{array}
                  \right], $$ $$ R_{i}(\lambda) := \left[
      \begin{array}{cccc}
        I_{(i-1)n} &  &  &  \\
          & 0_{n} & I_{n} &  \\
         & -I_{n} & P_{i}(\lambda) &  \\
          &  &  & I_{(d_A-i-1)n} \\
      \end{array}
    \right],
$$
$$ T_{i}(\lam) := \left[
      \begin{array}{cccc}
        0_{(i-1)n} &  &  &  \\
          & 0_{n} & \lam P_{i-1}(\lam) &  \\
         & \lam I_{n} & \lam^{2}P_{i-1}(\lam) &  \\
          &  &  & 0_{(d_A-i-1)n} \\
      \end{array}
    \right], $$
$$ D_{i}(\lambda) := \left[
      \begin{array}{cccc}
        0_{(i-1)n} &  &  &  \\
          & P_{i-1}(\lam) & 0_{n}  &  \\
         & 0_{n} & I_{n} &  \\
          &  &  & I_{(d_A-i-1)n} \\
      \end{array}
    \right], $$
and  $D_{d_A}(\lam) := \diag \left[0_{(d_A-1)n},  P_{d_A-1}(\lam)\right].
$
For simplicity, we often write $Q_{i}, R_{i}, T_{i}, D_{i}$ in place of $Q_{i}(\lam), R_{i}(\lam), T_{i}(\lam), D_{i}(\lam)$. Note that $D_{1}(\lambda) = M_{d_A}$, and $Q_{i}(\lam)$, $R_{i}(\lam)$ are unimodular for all $i = 1, \ldots, d_A-1$. Also note that $R_{i}^{\mathcal{B}}(\lam) = R_{i}(\lambda).$
\end{definition}

The auxiliary matrices satisfy the following relations.

\begin{lemma}[\cite{TDM}, Lemma~4.3]\label{lfqrtd}
Let $Q_{i}, R_{i}, T_{i}, D_{i}$ be as in Definition \ref{qrtdmfmp} and $M_i$'s be Fiedler matrices associated with $A(\lam).$ Then the following relations hold for $i = 1, \ldots, d_{A}-1$.
\begin{itemize}

\item[(a)] $Q_{i}^{\mathcal{B}}(\lambda D_{i})R_{i} = \lambda D_{i+1} + T_{i}$, and $Q_{i}^{\mathcal{B}}(M_{d_{A}-(i+1)}M_{d_{A}-i})R_{i} = M_{d_{A}-(i+1)} + T_{i}$.

\item[(b)] $R_{i}^{\mathcal{B}}(\lambda D_{i})Q_{i} = \lambda D_{i+1} + T_{i}^{\mathcal{B}}$, and $R_{i}^{\mathcal{B}}(M_{d_{A}-i}M_{d_{A}-(i+1)})Q_{i} = M_{d_{A}-(i+1)} + T_{i}^{\mathcal{B}}$.

\item[(c)] $T_{i}M_{j} = M_{j}T_{i} = T_{i}$ and $T_{i}^{\mathcal{B}}M_{j} = M_{j}T_{i}^{\mathcal{B}} = T_{i}^{\mathcal{B}}$ for all $j \leq d_{A}-i-2$.
\end{itemize}
\end{lemma}

\begin{definition} \label{qrtdmfmq}
Let $D(\lambda) = \sum\limits_{i = 0}^{d_D}\lambda^{i} D_{i}$ be an $m \times m$ matrix polynomial, and let $P_{i}(\lambda)$ be the degree $i$ Horner shift of $D(\lambda)$. For $1 \leq i \leq d_D-1$, define the following $m d_D \times m d_D$ matrix polynomials:
{\small $$
 Z_{i}(\lambda) := \left[
                    \begin{array}{cccc}
                      I_{(i-1)m} &  &  &  \\
                       & I_{m} & \lambda I_{m} &  \\
                       & 0_{m} & I_{m} &  \\
                       &  &  & I_{(d_D-i-1)m} \\
                    \end{array}
                  \right], $$ $$ J_{i}(\lambda) := \left[
      \begin{array}{cccc}
        I_{(i-1)m} &  &  &  \\
          & 0_{m} & I_{m} &  \\
         & I_{m} & P_{i}(\lambda) &  \\
          &  &  & I_{(d_D-i-1)m} \\
      \end{array}
    \right], $$

$$ H_{i}(\lam) := \left[
      \begin{array}{cccc}
        0_{(i-1)m} &  &  &  \\
          & 0_{m} & \lam P_{i-1}(\lam) &  \\
         & \lam I_{m} & \lam^{2}P_{i-1}(\lam) &  \\
          &  &  & 0_{(d_D-i-1)m} \\
      \end{array}
    \right], $$
$$ E_{i}(\lambda) := \left[
      \begin{array}{cccc}
        0_{(i-1)m} &  &  &  \\
          & P_{i-1}(\lam) & 0_{m}  &  \\
         & 0_{m} & I_{m} &  \\
          &  &  & I_{(d_D-i-1)m} \\
      \end{array}
    \right], $$}
and  $E_{d_D}(\lam) := \diag \left[0_{(d_D-1)m},  P_{d_D-1}(\lam)\right].
$
For simplicity, we often write $Z_{i}, J_{i}, H_{i}, E_{i}$ in place of $Z_{i}(\lam), J_{i}(\lam), H_{i}(\lam), E_{i}(\lam)$. Note that $E_{1}(\lambda) = M_{d_D}$, and $Z_{i}(\lam)$, $J_{i}(\lam)$ are unimodular for all $i = 1, \ldots, d_D-1$. Also note that $J_{i}^{\mathcal{B}}(\lam) = J_{i}(\lambda).$
\end{definition}

\begin{remark}\label{amfD}
Consider the auxiliary matrices $Z_{i}(\lambda), J_{i}(\lambda), H_{i}(\lambda)$, and $E_{i}(\lambda)$ given in Definition \ref{qrtdmfmq}. Then the Lemma \ref{lfqrtd} also holds for $Z_{i}(\lambda), J_{i}(\lambda), H_{i}(\lambda)$, and $E_{i}(\lambda)$.
\end{remark}

\begin{definition}[Auxiliary system polynomials] \label{amr}
Let $Q_{i}(\lambda), R_{i}(\lambda), T_{i}(\lambda)$, and $D_{i}(\lambda)$ be as in Definition \ref{qrtdmfmp}. Let $Z_{i}(\lambda), J_{i}(\lambda), H_{i}(\lambda)$, and $E_{i}(\lambda)$ be as in Definition \ref{qrtdmfmq}. Let $ d = \max\{d_A, d_D\},$ and $r= \min\{d_A, d_D\}$.
For $i = 1, \ldots, d-1$, define $(n d_A+m d_D) \times (n d_A+m d_D)$ system polynomials:
		
			\[\mathcal{Q}_i(\lambda)= \begin{cases} \left[\begin{array}{c|c}Q_i(\lambda)&0\\ \hline 0&Z_i(\lambda)\end{array}\right] ,& \text{for}\ 1\leq i \leq r-1 \\
			\left[\begin{array}{c|c}Q_i(\lambda)&0\\ \hline 0&I_{d_Dm}\end{array} \right],& \text{for}\ r\leq i \leq d-1 \ \text{and}\ d_A>d_D \\
			\left[\begin{array}{c|c}I_{d_An}&0\\ \hline 0&Z_i(\lambda)\end{array}\right], &\text{for}\ r\leq i \leq d-1 \ \text{and}\ d_A\ <d_D \end{cases} \]
		
		\[\mathcal{R}_i(\lambda)= \begin{cases} \left[\begin{array}{c|c}R_i(\lambda)&0\\ \hline 0&J_i(\lambda)\end{array}\right] ,& \text{for}\ 1\leq i \leq r-1 \\
			\left[\begin{array}{c|c}R_i(\lambda)&0\\ \hline 0&I_{d_Dm}\end{array} \right],& \text{for}\ r\leq i \leq d-1 \ \text{and}\ d_A>d_D \\
			\left[\begin{array}{c|c}I_{d_An}&0\\ \hline 0&J_i(\lambda)\end{array}\right], &\text{for}\ r\leq i \leq d-1 \ \text{and}\ d_A\ <d_D \end{cases} \]
		
		\[\mathcal{T}_i(\lambda)= \begin{cases} \left[\begin{array}{c|c}T_i(\lambda)&0\\ \hline 0&H_i(\lambda)\end{array}\right] ,& \text{for}\ 1\leq i \leq r-1 \\
			\left[\begin{array}{c|c}T_i(\lambda)&0\\ \hline 0&I_{d_Dm}\end{array} \right],& \text{for}\ r\leq i \leq d-1 \ \text{and}\ d_A>d_D \\
			\left[\begin{array}{c|c}I_{d_An}&0\\ \hline 0&H_i(\lambda)\end{array}\right], &\text{for}\ r\leq i \leq d-1 \ \text{and}\ d_A\ <d_D \end{cases} \]
	
	\[\mathcal{D}_i(\lambda)= \begin{cases} \left[\begin{array}{c|c}D_i(\lambda)&0\\ \hline 0&E_i(\lambda)\end{array}\right] ,& \text{for}\ 1\leq i \leq r-1 \\
		\left[\begin{array}{c|c}D_i(\lambda)&0\\ \hline 0&I_{d_Dm}\end{array} \right],& \text{for}\ \leq i \leq d-1 \ \text{and}\ d_A>d_D \\
		\left[\begin{array}{c|c}I_{d_An}&0\\ \hline 0&E_i(\lambda)\end{array}\right], &\text{for}\ r\leq i \leq d-1 \ \text{and}\ d_A\ <d_D \end{cases} \]
		\begin{eqnarray*}
\mbox{ and }\,\, \mathcal{D}_{d}(\lambda) & :=& \left[
                                  \begin{array}{c|c}
                                    D_{d_A}(\lam) & 0 \\
                                    \hline
                                    0 & E_{d_D}(\lam) \\
                                  \end{array}
                                \right], \text{ where } d = \max\{d_A, d_D\}.
\end{eqnarray*}
Note that $\mathcal{D}_{1}(\lambda) = \left[
                                          \begin{array}{c|c}
                                            D_{1}(\lambda) & 0 \\
                                            \hline
                                            0 & E_{1}(\lam) \\
                                          \end{array}
                                        \right] = \left[
                                                    \begin{array}{c|c}
                                                      M_{d_A} & 0 \\
                                                    \hline
                                                      0 & N_{d_D} \\
                                                    \end{array}
                                                  \right] = \mathbb{M}_{d}$ and that $\mathcal{Q}_{i}(\lambda)$ and $\mathcal{R}_{i}(\lambda)$ are unimodular matrix polynomials for $i = 1, \ldots, d-1$. Also, note that $\mathcal{R}_{i}^{\mathbb{B}}(\lambda) = \mathcal{R}_{i}(\lambda)$ for $i = 1, \ldots, d-1$.

\end{definition}

The auxiliary system polynomials satisfy the following relations.

\begin{lemma}
Let $\mathcal{Q}_{i}, \mathcal{R}_{i}, \mathcal{T}_{i}, \mathcal{D}_{i}$ be the system polynomials given in Definition \ref{amr} and $\mathbb{M}_i$'s be Fiedler matrices associated with $\mathcal{S}(\lam).$ Then the following system equivalence relations hold for $i = 1, \ldots, d-1$.
\begin{itemize}

\item[(a)] $\mathcal{Q}_{i}^{\mathbb{B}}(\lambda \mathcal{D}_{i})\mathcal{R}_{i} = \lambda \mathcal{D}_{i+1} + \mathcal{T}_{i}$, and $\mathcal{Q}_{i}^{\mathbb{B}}(\mathbb{M}_{d-(i+1)}\mathbb{M}_{d-i})\mathcal{R}_{i} = \mathbb{M}_{d-(i+1)} + \mathcal{T}_{i}$.

\item[(b)] $\mathcal{R}_{i}^{\mathbb{B}}(\lambda \mathcal{D}_{i})\mathcal{Q}_{i} = \lambda \mathcal{D}_{i+1} + \mathcal{T}_{i}^{\mathbb{B}}$, and $\mathcal{R}_{i}^{\mathbb{B}}(\mathbb{M}_{d-i}\mathbb{M}_{d-(i+1)})\mathcal{Q}_{i} = \mathbb{M}_{d-(i+1)} + \mathcal{T}_{i}^{\mathbb{B}}$.

\item[(c)] $\mathcal{T}_{i}\mathbb{M}_{j} = \mathbb{M}_{j}\mathcal{T}_{i} = \mathcal{T}_{i}$ and $\mathcal{T}_{i}^{\mathbb{B}}\mathbb{M}_{j} = \mathbb{M}_{j}\mathcal{T}_{i}^{\mathbb{B}} = \mathcal{T}_{i}^{\mathbb{B}}$ for all $j \leq d-i-2$.
\end{itemize}
\end{lemma}

\begin{proof}
\begin{itemize}

\item[(a)] We have
	\begin{align*}
		&\mathcal{Q}^{\mathbb{B}}_i (\lambda \mathcal{D}_i)\mathcal{R}_i =
		 \left[ \begin{array}{c|c}Q^{\mathcal{B}}_i&0\\ \hline 0&Z^{\mathcal{B}}_i\end{array}\right]\left[\begin{array}{c|c}\lambda D_i & \\ \hline  & \lambda E_i\end{array}\right] \left[\begin{array}{c|c}R_i&0\\ \hline 0&J_i\end{array}\right] \\
        &= \left[\begin{array}{c|c}Q^{\mathcal{B}}_i(\lambda D_i)R_i & \\ \hline  &Z^{\mathcal{B}}_i(\lambda E_i)J_i \end{array}\right]\\
        &= \left[\begin{array}{c|c}\lambda D_{i+1}+T_i& \\ \hline  & \lambda E_{i+1} + H_i\end{array}\right] \ (\text{By Lemma~\ref{lfqrtd}(a) and Remark~\ref{amfD}}) \\
        &=\left[\begin{array}{c|c}\lambda D_{i+1}& \\ \hline  & \lambda E_{i+1}\end{array}\right]+ \left[\begin{array}{c|c}T_i &  \\ \hline  &H_i\end{array}\right]
        =\lambda \mathcal{D}_{i+1} + \mathcal{T}_i \,\,\,\text{   and   }
    \end{align*}		
	 \begin{align*}
	 	&\mathcal{Q}^{\mathbb{B}}_i \mathbb{M}_{d-(i+1)}\mathbb{M}_{d-i}\mathcal{R}_i= \\&=\left[ \begin{array}{c|c}Q^{\mathcal{B}}_i&0\\ \hline 0&Z^{\mathcal{B}}_i\end{array}\right] \left[ \begin{array}{c|c}M_{d-(i+1)}& \\ \hline  &D_{d-(i+1)} \end{array}\right] \left[\begin{array}{c|c}M_{d-i}& \\ \hline  &D_{d-i} \end{array}\right] \left[\begin{array}{c|c}R_i&0\\ \hline 0&J_i\end{array}\right] \\
	 	&=\left[\begin{array}{c|c}Q^{\mathcal{B}}_iM_{d-(i+1)}M_{d-i}R_i&  \\ \hline  & Z^{\mathcal{B}}_iD_{m-(i+1)}D_{m-1}J_i\end{array}\right]\\
 	    &= \left[\begin{array}{c|c}M_{d-(i+1)}+T_i&  \\ \hline  &D_{d-(i+1)}+H_i\end{array}\right]\ (\text{Lemma~\ref{lfqrtd}(a) and Remark~\ref{amfD}}) \\ &=\left[\begin{array}{c|c}M_{d-(i+1)}& \\ \hline  & D_{d-(i+1)}\end{array}\right]+\left[\begin{array}{c|c}T_i& \\ \hline &H_i\end{array}\right]
 	    = \mathbb{M}_{d-(i+1)}+\mathcal{T}_i.
 	\end{align*}
\item[(b)] We have
  	\begin{align*}
  	&\mathcal{R}^{\mathbb{B}}_i (\lambda \mathcal{D}_i)\mathcal{Q}_i = \left[ \begin{array}{c|c}R^{\mathcal{B}}_i&0\\ \hline 0&J^{\mathcal{B}}_i\end{array}\right]\left[\begin{array}{c|c}\lambda D_i & \\ \hline  & \lambda E_i\end{array}\right] \left[\begin{array}{c|c}Q_i&0\\ \hline 0&Z_i\end{array}\right] \\
  	&= \left[\begin{array}{c|c}R^{\mathcal{B}}_i(\lambda D_i)Q_i & \\ \hline  &J^{\mathcal{B}}_i(\lambda E_i)Z_i \end{array}\right]\\
  	&= \left[\begin{array}{c|c}\lambda D_{i+1}+T^{\mathcal{B}}_i& \\ \hline  & \lambda E_{i+1} + H^{\mathcal{B}}_i\end{array}\right] \ (\text{From Lemma~\ref{lfqrtd}(b) and Remark~\ref{amfD}}) \\
  	&=\left[\begin{array}{c|c}\lambda D_{i+1}& \\ \hline  & \lambda E_{i+1}\end{array}\right]+ \left[\begin{array}{c|c}T^{\mathcal{B}}_i &  \\ \hline  &H^{\mathcal{B}}_i\end{array}\right]
  	=\lambda \mathcal{D}_{i+1} + \mathcal{T}^{\mathbb{B}}_i	   \,\,\,\text{   and   }
  	\end{align*}	
 	\begin{align*}
	&\mathcal{R}^{\mathbb{B}}_i \mathbb{M}_{d-i}\mathbb{M}_{d-(i+1)}\mathcal{Q}_i= \\ &\left[ \begin{array}{c|c}R^{\mathcal{B}}_i&0\\ \hline 0&J^{\mathcal{B}}_i\end{array}\right] \left[\begin{array}{c|c}M_{d-i}& \\ \hline  &D_{d-i} \end{array}\right] \left[ \begin{array}{c|c}M_{d-(i+1)}& \\ \hline  &D_{d-(i+1)} \end{array}\right] \left[\begin{array}{c|c}Q_i&0\\ \hline 0&Z_i\end{array}\right] \\
	&=\left[\begin{array}{c|c}R^{\mathcal{B}}_i M_{d-i}M_{d-(i+1)}Q_i&  \\ \hline  & J^{\mathcal{B}}_iD_{d-i}D_{d-(i+1)}Z_i\end{array}\right]\\
	&= \left[\begin{array}{c|c}M_{d-(i+1)}+T^{\mathcal{B}}_i&  \\ \hline  &D_{d-(i+1)}+H^{\mathcal{B}}_i\end{array}\right]\ (\text{By Lemma~\ref{lfqrtd}(b) and Remark~\ref{amfD}}) \\ &=\left[\begin{array}{c|c}M_{d-(i+1)}& \\ \hline  & D_{d-(i+1)}\end{array}\right]+\left[\begin{array}{c|c}T^{\mathcal{B}}_i& \\ \hline &H^{\mathcal{B}}_i\end{array}\right]
	= \mathbb{M}_{d-(i+1)}+\mathcal{T}^{\mathbb{B}}_i. 	
	\end{align*}	
\item[(c)] We have
	\begin{align*}
		\mathcal{T}_i\mathbb{M}_j &= \left[\begin{array}{c|c}T_i &  \\ \hline  & H_i \end{array}\right]\left[\begin{array}{c|c}M_j&  \\ \hline  &D_j \end{array}\right]
		= \left[\begin{array}{c|c}T_iM_j&  \\ \hline  & H_iD_j \end{array}\right]\\
		&= \left[\begin{array}{c|c}M_jT_i& \\ \hline  & D_jH_i\end{array}\right] \ (\text{by Lemma~\ref{lfqrtd}(c) and Remark~\ref{amfD}})\\
		&=\left[\begin{array}{c|c}M_j&  \\ \hline  &D_j \end{array}\right]\left[\begin{array}{c|c}T_i &  \\ \hline  & H_i \end{array}\right]
		=\mathbb{M}_j\mathcal{T}_i	
	\end{align*}	
	\begin{align*}
		\mathcal{T}^{\mathbb{B}}_i\mathbb{M}_j &= \left[\begin{array}{c|c}T^{\mathcal{B}}_i &  \\ \hline  & H^{\mathcal{B}}_i \end{array}\right]\left[\begin{array}{c|c}M_j&  \\ \hline  &D_j \end{array}\right]
		= \left[\begin{array}{c|c}T^{\mathcal{B}}_iM_j&  \\ \hline  & H^{\mathcal{B}}_iD_j \end{array}\right]\\
		&= \left[\begin{array}{c|c}M_jT^{\mathcal{B}}_i& \\ \hline  & D_jH^{\mathcal{B}}_i\end{array}\right] \ (\text{by Lemma~\ref{lfqrtd}(c) and Remark~\ref{amfD}})\\
		&=\left[\begin{array}{c|c}M_j&  \\ \hline  &D_j \end{array}\right]\left[\begin{array}{c|c}T^{\mathcal{B}}_i &  \\ \hline  & H^{\mathcal{B}}_i \end{array}\right]
		=\mathbb{M}_j\mathcal{T}^{\mathbb{B}}_i.
	\end{align*}	
\end{itemize}
\end{proof}

\begin{definition}\label{syspendef}
Let $\mathbb{L}_{\sigma}(\lambda) = \lambda \mathbb{M}_{d} - \mathbb{M}_{\sigma}$ be the Fiedler pencil of $\mathcal{S}(\lambda)$ given in (\ref{smpq}) associated with a bijection $\sigma$. For $j = 1, 2, \ldots, d$, define
$$\mathbb{M}_{\sigma}^{(j)} := \prod_{\sigma^{-1}(i)\leq d-j}\mathbb{M}_{\sigma^{-1}(i)}, $$ where the factors $\mathbb{M}_{\sigma^{-1}(i)}$ are in the same relative order as they are in $\mathbb{M}_{\sigma}$. Note that $\mathbb{M}_{\sigma}^{(1)} = \prod_{\sigma^{-1}(i)\leq d-1}\mathbb{M}_{\sigma^{-1}(i)} = \mathbb{M}_{\sigma}$  and that $\mathbb{M}_{\sigma}^{(d)} = \mathbb{M}_{0}$. Also for $j = 1, 2, \ldots, d$ define the $(nd_{A}+md_{D})\times (nd_{A}+md_{D})$ system pencils $\mathbb{L}_{\sigma}^{(j)}(\lambda) := \lambda \mathcal{D}_{j}(\lambda) - \mathbb{M}_{\sigma}^{(j)}. $
 Observe that $\mathbb{L}_{\sigma}^{(1)}(\lambda) = \lambda \mathcal{D}_{1} - \mathbb{M}_{\sigma}^{(1)} = \lambda \mathbb{M}_{m} - \mathbb{M}_{\sigma} = \mathbb{L}_{\sigma}$ and that $$
\mathbb{L}_{\sigma}^{(d)}(\lambda) = \lambda \mathcal{D}_{d} - \mathbb{M}_{\sigma}^{(d)}
= \lambda \left[
            \begin{array}{c|c}
              D_{d_A} & 0 \\
             \hline
             0 & -E_{d_D} \\
            \end{array}
          \right] - \mathbb{M}_{0} $$
          $$ =  \left[
            \begin{array}{cc|cc}
              -I_{(d_A-1)n} &  & -(e_{d_A} e_{d_D}^{T}) \otimes B & \\
                & A(\lam) &  & \\
             \hline
              (e_{d_D} e_{d_A}^{T})\otimes C &  & -I_{(d_D-1)m} & \\
               &  & & D(\lam) \\
            \end{array}
          \right]. $$

\end{definition}

The next result shows that  $\mathbb{L}_{\sigma}^{(i)}(\lambda) \thicksim_{se} \mathbb{L}_{\sigma}^{(i+1)}(\lambda)$ for $i=1, 2, \ldots, d-1.$

\begin{lemma} \label{lietli+1} We have
$\mathbb{L}_{\sigma}^{(i)}(\lambda)\thicksim_{se} \mathbb{L}_{\sigma}^{(i+1)}(\lambda)$ for $i = 1, 2, \ldots, d-1$. More precisely,  if $\mathcal{Q}_{i}$ and $\mathcal{R}_{i}$ are the system polynomials given in Definition~\ref{amr}, then
$$\mathbb{L}_{\sigma}^{(i+1)}(\lambda)=
\begin{cases}
\mathcal{Q}_{i}^{\mathbb{B}} \mathbb{L}_{\sigma}^{(i)}(\lambda) \mathcal{R}_{i}, &  \text{ if } \sigma \text{ has a consecution at } d-i-1 \\
\mathcal{R}_{i}^{\mathbb{B}} \mathbb{L}_{\sigma}^{(i)}\mathcal{Q}_{i},  & \text{ if } \sigma \text{ has an inversion at } d-i-1.
\end{cases}$$
\end{lemma}

\begin{proof} The proof is exactly the same as that of Lemma~4.5 in \cite{TDM}.
\end{proof}

It is now immediate that  a Fiedler pencil is a Rosenbrock linearization of $\mathcal{S}(\lam).$

\begin{theorem}[Rosenbrock linearization] Let $\mathcal{S}(\lam)$ be an $(n+m)\times (n+m)$ system polynomial (regular or singular) given in (\ref{smpq}).
 Then a Fiedler pencil $\mathbb{L}_{\sigma}(\lambda)$ of the system polynomial $\mathcal{S}(\lam)$ is a Rosenbrock linearization of $\mathcal{S}(\lam)$.
\end{theorem}

\begin{proof} By Lemma~\ref{lietli+1}, we have $d-1$ system equivalences
$$
\mathbb{L}_{\sigma}(\lambda) = \mathbb{L}_{\sigma}^{(1)}(\lambda) \thicksim_{se} \mathbb{L}_{\sigma}^{(2)}(\lambda) \thicksim_{se}\cdots \thicksim_{se} \mathbb{L}_{\sigma}^{(d)}(\lambda)$$
$$ =  \left[
            \begin{array}{cc|cc}
              -I_{(d_{A}-1)n} &  & -(e_{d_{A}} e_{d_{D}}^{T}) \otimes B & \\
                & A(\lam) &  & \\
             \hline
              (e_{d_{D}} e_{d_{A}}^{T})\otimes C &  & -I_{(d_{D}-1)m} & \\
               &  & & D(\lam) \\
            \end{array}
          \right], $$
where $\mathbb{L}_{\sigma}^{(i)}(\lambda)$ is as in Lemma~\ref{lietli+1}.  This shows that $\mathbb{L}_{\sigma}(\lambda) \thicksim_{se} I_{(d_{A}-1)n} \oplus \mathcal{S}(\lam) \oplus I_{(d_{D}-1)m}.$ \end{proof}

%The  $m-1$ system equivalences in (\ref{uefl}) provide the desired system equivalence that transforms $\mathbb{L}_{\sigma}(\lambda)$ to the extended system polynomial $I_{(m-1)n}\oplus \mathcal{S}(\lam)$  and can be constructed as follows.

\begin{corollary}\label{eouavflor}
Let $\mathbb{L}_{\sigma}(\lambda)$ be the Fiedler pencil of $\mathcal{S}(\lambda)$ given in (\ref{smpq}) associated with a bijection $\sigma$, and $\mathcal{Q}_{i}, \mathcal{R}_{i}$ for $i = 1, 2, \ldots d-1$, be as in Definition \ref{amr}. Then
\begin{equation*}
\mathcal{U}(\lambda) \mathbb{L}_{\sigma}(\lambda)\mathcal{V}(\lambda) = \left[
            \begin{array}{cc|cc}
              -I_{(d_{A}-1)n} &  & -(e_{d_{A}} e_{d_{D}}^{T}) \otimes B & \\
                & A(\lam) &  & \\
             \hline
              (e_{d_{D}} e_{d_{A}}^{T})\otimes C &  & -I_{(d_{D}-1)m} & \\
               &  & & D(\lam) \\
            \end{array}
          \right]
\end{equation*}
$$ \thicksim_{se} I_{(d_{A}-1)n} \oplus \mathcal{S}(\lam) \oplus I_{(d_{D}-1)m},$$
where $\mathcal{U}(\lambda)$ and $\mathcal{V}(\lambda)$ are $(nd_A+m d_{D})\times (nd_A+m d_{D})$ unimodular system polynomials given by
$$\mathcal{U}(\lambda) := \mathcal{U}_{0}\mathcal{U}_{1}\cdots \mathcal{U}_{d-3}\mathcal{U}_{d-2}, \mbox{  with  } \mathcal{U}_{i} =
\begin{cases}
\mathcal{Q}_{d-(i+1)}^{\mathbb{B}}, & \text{if } \sigma \text{ has a consecution at } i, \\
\mathcal{R}_{d-(i+1)}^{\mathbb{B}}, & \text{if } \sigma \text{ has an inversion at } i,
\end{cases}
$$
$$\mathcal{V}(\lambda) := \mathcal{V}_{d-2}\mathcal{V}_{d-3}\cdots \mathcal{V}_{1}\mathcal{V}_{0}, \mbox{  with  } \mathcal{V}_{i} =
\begin{cases}
\mathcal{R}_{d-(i+1)}, & \text{if } \sigma \text{ has a consecution at } i, \\
\mathcal{Q}_{d-(i+1)}, & \text{if } \sigma \text{ has an inversion at } i.
\end{cases}
$$
\end{corollary}

The indexing of $\mathcal{U}_i$ and $\mathcal{V}_i$ factors in $\mathcal{U}(\lam)$ and $\mathcal{V}(\lam),$ respectively,  in Corollary~\ref{eouavflor} has been chosen for simplification of notation and has no other special significance.

\begin{remark}
If we consider $D(\lam)$ is a matrix polynomial of degree $1$ then the Fiedler pencils $\mathbb{L}_{\sigma}(\lambda)$ are linearizations of the system matrix of LTI state-space system, see \cite{rafinami}.
\end{remark}

\begin{remark} Consider the system matrix $\mathcal{S}(\lambda)$ and associated transfer function $R(\lam)$ given in (\ref{smpq}) and (\ref{tf}), respectively.
Given an eigenvector $x$ of $\mathbb{L}_{\sigma}(\lambda)$ one can determine an eigenvector of $\mathcal{S}(\lambda)$ from $x$. That is, one can recover eigenvectors of $R(\lam)$ and $\mathcal{S}(\lambda)$ from those of the Fiedler pencils  of $R(\lam)$. It is directly follows from the Theorem~4.10 and Theorem~4.11 in \cite{beh21}.
\end{remark}

\section{Conclusions} We have considered a  multivariable state-space system and its associated system matrix $\mathcal{S}(\lam).$  We have introduced Fiedler pencils of $\mathcal{S}(\lam)$ and described
an algorithm for their construction. Finally, we have shown that Fiedler pencils are linearizations of $\mathcal{S}(\lam)$.

\bibliographystyle{plain}
\bibliography{Beh_m}

\end{document}